\documentclass[a4paper,12pt]{amsart}
\usepackage{amsmath,amssymb,amscd,amsfonts, amsthm}
\usepackage{amsrefs, esint}
\numberwithin{equation}{section}
\newtheorem{theorem}{Theorem}[section]

\newtheorem{rem}[theorem]{Remark}
\newtheorem{defn}[theorem]{Definition}
\newtheorem{lemma}[theorem]{Lemma}

\newtheorem{prop}[theorem]{Proposition}

\newtheorem{corollary}[theorem]{Corollary}

\def\det{\mathop{\rm det}\nolimits}

\def\Re{\mathop{\rm Re}\nolimits}

\def\tr{\mathop{\rm tr}\nolimits}

\def\dbar{\bar\partial}
\def\ddbar{\partial\bar\partial}
\def\d{\partial}
\def\la{\lambda}

\def\cC{{\mathcal C}}

\def\cZ{{\mathcal Z}}

\def\cF{{\mathcal F}}
 
\let\sm=\setminus
\let\ol=\overline
\let\ul=\underline

\let\ep=\varepsilon
\let\vp=\varphi 

\def\bC{{\mathbb C}}
\def\bR{{\mathbb R}}

\def\a{{\alpha}}
\def\z{{\zeta}}

\def\de{{\delta}}

\def\tg{\Gamma}
\def\g{\gamma}

\def\b{\beta}

\def\k{{\kappa}}

\def\tr{{\text{tr}}}

\title[Conic cscK metrics]
{Conic singularities metrics\\ with prescribed scalar curvature: \\ a priori estimates for normal crossing divisors}

\author{Long LI}
\author{Jian WANG}
\author{Kai ZHENG}
\address{Institute Fourier, 100 rue des maths 38610 Gi\`eres, Grenoble, France}
\email{Long.Li1@univ-grenoble-alpes.fr}
\address{Institute Fourier, 100 rue des maths 38610 Gi\`eres, Grenoble, France}
\email{jian.wang1@univ-grenoble-alpes.fr}
\address{Mathematics Institute, University of Warwick, Coventry, CV4 7AL, UK}
\email{K.Zheng@warwick.ac.uk}  

\begin{document}
\maketitle 

$$\text{Dedicated to Prof. Jean-Pierre Demailly} $$

\begin{abstract}
The purpose of this paper is to prove the a priori estimates for constant scalar curvature K\"ahler metrics with
conic singularities along normal crossing divisors. The zero order estimates are proved by a reformulated version of Alexandrov's maximum principle. 
The higher order estimates follow from Chen-Cheng's frame work, equipped with new techniques to handle the singularities. 
Finally, we extend these estimates to the twisted equations.
\end{abstract}

\section{Introduction}
\label{sec-001}

Recently Chen-Cheng (\cite{CC1}, \cite{CC2}, \cite{CC3}) established the a priori estimates for the constant scalar curvature K\"ahler(cscK) metrics equation, which are fundamental towards the Yau-Tian-Donaldson conjecture on the existence of the cscK metrics. Their estimates lead to the resolution of the properness conjecture and Donaldson's geodesic stability conjecture.



Our  goal is to prove a singular version of the Yau-Tian-Donaldson conjecture,
and this first paper aims to generalise Chen-Cheng's a priori estimates to the \emph{log-smooth klt pair}.
That is to say, our metrics develop cone like singularities along normal crossing divisors.
In the subsequent papers, we will discuss the existence problem for cscK metrics
on \emph{singular klt pairs}.

Let $(X,D)$ be a log smooth klt pair, where $D: = \sum_{k=1}^d (1-\b_k) D_k$
is an $\bR$-divisor on the compact K\"ahler manifold $X$.
Here the index $\b: = \{ \b_k \}_{k=1}^d$ is a collection of angles $0< \b_k <1$.
For some $0< \a <  \min_k\{ \frac{1}{\b_k} -1 ,1\}$, 
we consider a conic H\"older space $\cC^{2,\a,\b}$ first introduced by Donaldson \cite{Don12}.

Suppose $(\vp, F)\in \cC^{2,\a,\b}$ is a pair satisfying the conic cscK equation (Defintion (\ref{csck-defn-001})).
Denote $H_\b(\vp)$ by the entropy of a conic K\"ahler potential $\vp$, with respect to the background metric $\omega_\b$ (equation (\ref{csck-001})): 
$$ H_\b(\vp): = \int_X \log \frac{\omega_\vp^n}{\omega_\b^n}\omega_\vp^n.$$ 
Then the following estimates are proved.

\begin{theorem}
\label{intr-thm-001}
Let $(\vp, F)$ be a $\cC^{2,\a,\b}$-conic cscK pair on $(X,D)$. 
Suppose that its entropy $H_\b(\vp) $ is bounded by a uniform constant $C$.
Then there exists another uniform constant $\tilde C$, such that the following holds:
\begin{enumerate}
\item[(i)] the $C^0$-estimate
$$|| \vp ||_0 < \tilde C;$$

\item[(ii)] the non-degeneracy estimate
$$ -\tilde C < F < \tilde C; $$

\item[(iii)]  the gradient $F$-estimate and the $C^2$-estimate
$$
\max_X |\nabla_\vp F|_\vp + \max_X \text{tr}_{\omega_\b}\omega_\vp < \tilde C. 
$$
\end{enumerate}
\end{theorem}

The $C^0$ estimate is proved in Theorem (\ref{main-thm-001}) and Corollary (\ref{c0-cor-001}),
and the non-degneracy estimate is proved in Lemma (\ref{nd-lem-001}). 
Comparing with Chen-Cheng's work \cite{CC1}, the new difficulty is that 
Alexandrov's maximum principle(AMP) fails in the conic case. More precisely, 
the constant appearing in Chen-Cheng's estimate depends on the diameter of the coordinate ball,
on which we applied this maximum principle. 
However, the diameter has to become smaller and smaller when the ball is approaching the divisor,
and then we lose the control of the constant. 

In order to overcome this difficulty, we developed a new version of AMP, the \emph{Generalised Alexandrov's maximum principle}(GAMP) in Theorem (\ref{amp-thm-001}). 
The key observation is that this maximum principle still works for a function $u$ 
if the upper contact set $\tg^+_u$ of this function is completely disjoint from the singular locus of $u$.
Therefore, we can utilise this new maximum principle in the estimates, 
by adding an extra ``extremely" pseudo-convex auxiliary function near the divisor.


The integral method on compact manifold (iteration without assuming uniform Sobolev constant on varing metrics) 
from Chen-He \cite{CH} is important in Chen-Cheng's work. 
Following this basic frame work, the gradient $F$-estimate and the $C^2$-estimate in the conic setting are also proved via 
the following $W^{2,p}$ type estimate. 
\begin{theorem}[Theorem (\ref{main-thm-2})]
\label{intr-thm-002}
Let $(\vp, F)$ be a $\cC^{2,\a,\b}$-conic cscK pair on $(X,D)$.
 For any $1< p < +\infty$,
there exists a uniform constant $C'$ such that 
$$ \int_X (\tr_{\omega_\b}\omega_\vp)^p \omega_\b^n < C'. $$
Here this constant $C'$ depends on $p$, $|| \vp ||_0$, $|| F ||_0$ and conic background metrics $\omega_\b, \omega_D$ on $(X,D)$.
\end{theorem}
Here $\omega_D$ is another conic background metric introduced by Donaldson \cite{Don12},
and the constant $C'$ depends on both $\omega_\b$ and $\omega_D$. 
The reason is that we need to switch the background metrics during the proof of the $W^{2,p}$-estimate,
but this does no harm to our $L^p$-norm $|| n+ \Delta\vp ||_{L^p(\omega_\b^n)}$ 
since $\omega_D$ and $\omega_\b$ are quasi-isometric on $X$.

For later purpose, we assumed that the cscK pair $(\vp, F)$ lies in the conic H\"older space $\cC^{2,\a,\b}$. 
In practice, all these a priori estimates still hold, if we  only assume $(\vp, F)\in \cC^{1,\bar 1}_\b$(Definition (\ref{pe-defn-001})) in the very beginning.

In order to investigate the existence of the conic cscK metrics, we are further led to studying the following continuity path on $Y:=X\sm\text{Supp}(D)$:
$$ t(R_{\vp} - \ul{R}_\b) = (1-t) (\tr_{\vp}\tau_\b - \ul\tau_\b), $$
for $t\in [0,1]$. Here $\tau$ is a closed $(1,1)$ form varying in a fixed K\"ahler class. 
More precisely, we assumed 
$ \tau: = \tau_0 + dd^c f \geq 0$,
for some fixed smooth $(1,1)$ form $\tau_0$ on $Y$ with $|\tau_0|_{\omega_\b}$ uniformly bounded,
and the function $f$ satisfies
$$\sup_X f = 0; \ \ \  \ \ \ \int_X e^{-p_0f} \omega_\b^n < +\infty, \ \  \emph{for some $p_0>1$}.$$

With these constraints, a triple $(\vp, F, f)\in\cC^{2,\a,\b}$ is the solution to the twisted conic-cscK equation if they satisfy equations (\ref{tw-001}) and (\ref{tw-002}).
Then we extend our estimates to the following. 
\begin{theorem}
\label{intr-thm-003}
Let $(\vp, F, f)\in\cC^{2,\a,\b}$ be a triple of the twisted equations.
Suppose the entropy $H_\b(\vp) $ is bounded by a uniform constant $C$.
Then there exists another uniform constant $C''$, such that the following holds:
\begin{enumerate}
\item[(iv)] the $C^0$-estimate
$$|| \vp ||_0 < C'';$$

\item[(v)] the non-degeneracy estimate
$$  -C'' < F < C''; $$

\item[(vi)] there exists a constant $k_n$, only depending on the dimension $n$,
such that if $p_0 > k_n$, then we have 
$$
 |\nabla_\vp (F+f)|_\vp < C''.
$$
\end{enumerate}
\end{theorem}

Since the upper bound of the $(1,1)$ form $\tau$ is out of control in the twisted case, 
we no longer expect the $C^2$-estimate directly. However, the $C^2$-estimate can be actually deduced from the gradient $F$-estimate, 
by a conic version of Chen-He's integral estimate.

Similarly, the gradient $F$-estimate is proved via the following $W^{2,p}$-estimate. 
\begin{theorem}[Theorem (\ref{tw-thm-001})]
\label{intr-thm-004}
Let $(\vp, F, f)\in\cC^{2,\a,\b}$ be a triple of the twisted equations.
For any $p\geq 1$ there exists a constant $\hat C$ satisfying 
$$ \int_X e^{-(p-1)f} (\tr_{\omega_\b} \omega_\vp)^p \omega_\b^n \leq \hat C, $$
Here the constant $\hat C$ depends on $p$, $p_0$(uniform if $p_0$ is bounded away from $1$), 
$||\vp||_0$, $|| F + f ||_0$, and background metrics $\omega_\b$ and $\omega_D$.
\end{theorem}

More generally, our zero order estimates (the $C^0$ and non-degeneracy estimates) can be also used on singular klt pairs.
In fact, after pulling back to a log-resolution, the metric has 
conic singularities along normal crossing divisors, but it is possibly degenerate along some exceptional divisors. 
Therefore, we can apply our tricks on the resolution, and then the estimates follow from GAMP again. 

Furthermore, the higher order estimates on singular klt pairs, like the $W^{2,p}$-estimate and $C^2$-estimate, 
can also be realised on a compact domain away from the divisor. These topics will be discussed in a sequel paper.

$\mathbf{Acknowledgement}$: 
We want to show our great thanks to Prof. Xiuxiong Chen, for he introduced this problem to us.
The first author is grateful to Prof. Donaldson for sharing his beautiful insights on conic K\"ahler geometry,
and he would also like to thank Prof. Demailly, Prof. M. P\u aun, Prof. Berndtsson, Prof. S. Boucksom, and Prof. Guedj 
for lots of useful discussions and continuous  encouragement.
The second author would like to thank Prof. Besson for his encouragement.

The first author is supported by the ERC-ALKAGE project.
The second author is supported by the ERC-GETOM and ANR-CCEM projects.
The third author has received funding from the European Union's Horizon 2020 research and innovation programme under the Marie Sklodowska-Curie grant agreement No. 703949.

\section{Preliminary }
\label{sec-002}
Let $(X,\omega)$ be an $n$-dimensional compact complex K\"ahler manifold.
Suppose $D:= \sum_{k=1}^d (1-\b_k) D_k$ is an $\mathbb{R}$-divisor on $X$ 
with simple normal crossing support such that the angle $\b_k\in (0,1)$ for all $k$.
Then $(X,D)$ is called as a \emph{log smooth klt pair}. 

Near a point $p$ on the support of $D$, there exists a holomorphic coordinate system $\{z_i\}$ such that  
the support $\text{Supp}(D)$ is defined by the equation $\{z_1\cdots z_d =0\}$.
Then a model conic metric $\omega_{cone}$ with cone angle $\b_k$ along $D_k$ can be written as 
\begin{equation}
\label{p-001}
\omega_{cone}: = \sum_{k=1}^d \frac{\sqrt{-1} dz_k \wedge d\bar z_k}{|z_k|^{2-2\b_k}} + \sum_{k=d+1}^n \sqrt{-1} dz_k\wedge d\bar z_k.
\end{equation}

A positive current $\omega_\vp: = \omega + dd^c\vp $ is a conic K\"ahler metric with cone angle $\b_k$ along $D_k$,
if it is smooth on $X\sm (\bigcup D_k)$ and quasi-isometric to the model metric $\omega_{cone}$ near each point $p\in \text{Supp}(D)$,
i.e. it satisfies 
$$ C^{-1}\omega_{cone} \leq \omega_{\vp} \leq C \omega_{cone}, $$
for some constant $C>0$.

When the divisor $D$ is a smooth hypersurface, Donaldson \cite{Don12} introduced the conic H\"older spaces for the potential 
like $\vp\in\cC^{0,\a,\b}$ or  $\vp\in\cC^{2,\a,\b}$, for some constant $\a\in(0,1)$ with $ \a\b < 1-\b $.
Moreover, he proved a version of the Schauder estimate(\cite{Don12}, \cite{Brendle}) for the conic Laplacian operator. 

\subsection{Conic K\"ahler-Einstein metrics}
Let $(L_k, \phi_k), 1\leq k \leq d $ be a set of hermitian  line bundles, with non-trivial sections $s_k\in H^0(X, L_k)$.
Assume the divisors $D_k: = \{s_k =0 \}$ are smooth, and they have strictly normal intersections. 
For simplicity, we write the norm of the sections as $|s_k|^2: = |s_k|^2 e^{-\phi_k}$.
Then a simple example of conic K\"ahler metrics, the \emph{Donaldson metric},
can be written as 
$$ \omega_D: = \omega +   \frac{1}{N} \sum_{k=1}^d dd^c |s_k|^{2\b_k}, $$
for some $N>0$ large. 

This example has been widely used as the background metric in the study of conic geometric equations.
In fact, there exists a natural smooth approximation of $\omega_D$ as 
$$ \omega_{D, \ep}: = \omega + \frac{1}{N} \sum_{k=1}^d dd^c (|s_k|^2 + \ep^2)^{\b_k},$$
for every $\ep >0$ small. However, the holomorphic bisectional curvature of this approximation $R_{i\bar i j \bar j}(\omega_{D,\ep})$
grows too fast along certain directions near the divisor. 
Therefore, Campana-Guenancia-P\u aun \cite{CGP13} and Guenancia-P\u aun \cite{GP}
introduced another smooth approximation as 
$$ \tilde\omega_{D,\ep}: = \omega + \frac{1}{N} \sum_{k=1}^d dd^c \chi_k (|s_k|^2 + \ep^2), $$
where the auxiliary function $\chi_k(\ep^2+t)$ is a smooth perturbation of the function $(\ep^2 +t)^{\b_k}$.
This is a ``better" choice in the sense that the bisectional curvature $R_{i\bar i j\bar j}(\tilde\omega_{D,\ep})$ 
has a slower growth rate near the divisor. 

In the work \cite{CGP13} and \cite{GP}, they studied the regularities of the so called \emph{conic K\"ahler-Einstein}(KE) metrics as 
\begin{equation}
\label{cke-001}
(\omega + dd^c\vp)^n = e^{f+\lambda \vp}d\mu_D,
\end{equation}
where $\lambda = \{-1,0,1\}$, $f\in C^{\infty}(X)$, and the measure $\mu_D$ is defined by 
$$d\mu_D: = \frac{\omega^n}{\prod_{k=1}^d |s_k|^{2-2\b_k}}.$$
Here $e^f d\mu_D$ is a probability measure if $\lambda=0$. 
They actually proved that an $L^{\infty}$ solution $\vp$ of equation (\ref{cke-001}) is  always in the space $\cC^{2,\a,\b}$.

For the case $\lambda \geq 0$,
by the celebrated work of Yau \cite{Yau} on Calabi's conjecture, 
there always exists a smooth approximation for the conic KE metric as follows
\begin{equation}
(\omega + dd^c\vp_{\ep})^n = \frac{e^{\lambda\vp_{\ep} + f + c_{\ep}} \omega^n}{\prod_{k=1}^d (|s_k|^2 + \ep^2)^{1-\b_k}},
\end{equation}
for some uniformly bounded constant $c_{\ep}$.

According to the expansion formula of the conic KE metric (\cite{YZ}, \cite{JMR}),
the bisectional curvature $R_{i\bar i j\bar j} (\omega_{\vp})$ ($R_{i \bar i j\bar j}(\omega_{\vp_\ep})$) 
behaves even better than $R_{i\bar i j\bar j} (\omega_{D})$ ($R_{i\bar i j\bar j} (\tilde\omega_{D,\ep})$) near the divisor,
when the divisor is smooth.
As inspired from the third author's previous work \cite{Zheng1},
we will use a special conic KE metric as the background metric.

\subsection{Conic cscK metrics}

For $\lambda =0$, there always exists a solution $\omega_\b: = \omega + dd^c\psi_\b$ for the \emph{conic Calabi-Yau} equation as 
\begin{equation}
\label{csck-000}
( \omega + dd^c\psi_\b)^n = \frac{e^{f} \omega^n }{\prod_{k=1}^d |s_k|^{2-2\b_k}},
\end{equation}
where $\b:= (\b_1,\cdots, \b_k)$ is a collection of angles.  
In other words, it solves the following geometric equation 
\begin{equation}
\label{csck-001}
Ric (\omega_\b) =  \Theta + \sum_{k=1}^d (1-\b_k) [D_k],
\end{equation}
where $\Theta$ is a smooth closed $(1,1)$ form on $X$ defined by 
$$\Theta: =  -dd^c f + Ric(\omega) - \sum_{k=1}^d (1-\b_k) dd^c\phi_k. $$

Let $\omega_{\vp}: = \omega_\b + dd^c\vp $ be a conic K\"ahler metric 
with cone angle $\b_k$ along each $D_k$. Suppose this conic metric is a solution of the following two coupled equations
\begin{equation}
\label{csck-002}
(\omega_{\b} + dd^c\vp)^n = e^{F} \omega_{\b}^n,
\end{equation}
\begin{equation}
\label{csck-003}
\Delta_{\vp}F= - \underline{R}_\b + tr_{\vp} \Theta,
\end{equation}
where $\underline{R}_{\b}$ is a topological constant depending on the angle $\b$,
and we assume the normalisation $\int_X e^F \omega_{\b}^n = 1$.
Observe that the solution $\omega_{\vp}$ has constant scalar curvature ($ R_{\vp} = \underline{R}_{\b}$),
outside the support of the divisor $D$.  

If the solution $\vp$ is in the space $\cC^{2,\a,\b}$,
then $F$ is in $\cC^{0,\a,\b}$ by equation (\ref{csck-002}).
When the divisor $D$ is smooth,
we further have $F\in\cC^{2,\a,\b}$ by equation (\ref{csck-003}) and Donaldson's Schauder estimate. 
Therefore, it makes sense to assume that $\vp$ and $F$ always have the same regularities in general. 

\begin{defn}
\label{csck-defn-001}
A pair of functions $(\vp, F)$ is called a $\cC^{2,\a,\b}$-conic cscK pair on $(X,D)$, 
if the potential $\vp$ is in the space $\cC^{2,\a,\b}$,
and its associated K\"ahler metric $\omega_{\vp}$ (with cone angle $\b_k$ along each $D_k$)
satisfies the coupled equations 
(\ref{csck-002}), (\ref{csck-003}) on $X\sm (\bigcup D_k)$, with $F\in \cC^{2,\a,\b}$.
\end{defn}

Since we switch the background metric to $\omega_{\b}$, 
the potential space has a one-one correspondence with the previous one, i.e. $\tilde\vp: = \psi_\b + \vp$.
However, the conic H\"older space is unchanged, since $\psi_{\b}$ is also in the space $\cC^{2,\a,\b}$ (\cite{GP}). 
Therefore, we can stick to this new potential space as the collection of 
all $\omega_\b$-plurisubharmonic functions with $\cC^{2,\a,\b}$ regularities.

When the divisor is smooth, the higher regularities have been 
known in \cite{LZ} for $0< \b < \frac{1}{2}$, 
and in (\cite{YZ}, \cite{Zheng}) for any angles.
There we used the model cone metric as our background metric. 
However, this is not an issue, since both potential of $\omega_D$ and $\omega_{\b}$ have higher regularities.

\section{Generalised Alexandrov's maximum principle}
Let $\Omega$ be a bounded open domain in $\bR^n$, with smooth boundary $\d\Omega = \ol{\Omega}\bigcap (\bR^n \sm \Omega)$.
Let $L$ be a second order differential operator:
$$L = \sum_{i,j=1}^n a_{ij}(x) D_{ij} + \sum_{i=1}^n b_i(x) D_i + c(x),$$
with $a_{ij}\in L^{\infty}_{loc}(\Omega)$ and $b_i, c\in L^{\infty}(\Omega)$.
Moreover, we assume $a_{ij} = a_{ji}$. 

The operator $L$ is called \emph{elliptic} on $\Omega$ if for every $x\in\Omega$ there exists $\lambda(x) >0$, such that 
$$ \sum_{i,j=1}^n a_{ij} \xi_i\xi_j \geq \lambda(x) |\xi|^2,$$
for all $\xi\in \bR^n$. 
Moreover, for elliptic operator $L$ one defines 
$$ \mathfrak{D}^*: = (\det (a_{ij}))^{1/n}. $$

For any continuous function $u$ on the set $\ol{\Omega}$, 
we can introduce the \emph{upper contact set} of $u$, 
which is roughly speaking the set of points in $\Omega$ that have a tangent plane above the graph of $u$. 
\begin{defn}
\label{amp-defn-001}
For any $u\in C(\ol{\Omega})$, the upper contact set $\tg^+$ is defined by 
$$\tg^+:= \left\{ y\in \Omega;\  \exists\  p_y\in \bR^n\ \emph{such that}\ \forall x\in \Omega:  \  u(x) \leq u(y) + p_y\cdot (x-y) \right\} $$
\end{defn}

The set $\tg^+$ is relatively closed in $\Omega$.
If $u\in C^1(\Omega)$, then $p_y = \nabla u (y)$ for any $y\in \tg^+$.

Moreover, if $u\in C^2(\Omega)$, then the Hessian matrix $(D_{ij}u)$ is semi-negative on $\tg^+$.
In other words, the set $\tg^+$ consists of all ``concave points" of $u$.

Then we invoke Alexandrov's maximum principle(AMP) as follows (\cite{GT}, \cite{Sweers}). 
\begin{theorem}
\label{amp-thm-001}
Let $\Omega$ be bounded and $L$ elliptic with $c\leq 0$. Suppose that $u\in C^2(\Omega)\bigcap C(\ol{\Omega})$ satisfies $Lu\geq f$ with
$$ \frac{|b|}{\mathfrak{D}^*}, \frac{f}{\mathfrak{D}^*}\in L^n(\Omega), $$
and then one has 
\begin{equation}
\label{amp-001}
\sup_{\Omega} u \leq \sup_{\d\Omega} u^+ + C\cdot \text{diam}(\Omega) \cdot \left\Vert \frac{f^-}{\mathfrak{D}^*}\right\Vert_{L^n (\Gamma^+)},
\end{equation}
with
$$  C: = C \left( n, \left\Vert \frac{|b|}{\mathfrak{D}^*} \right\Vert_{L^n(\tg^+)} \right). $$
\end{theorem}

When the function $u$ is no longer $C^2$ in $\Omega$, the Alexandrov maximum principle fails in general. 
However, observe that the RHS of inequality ($\ref{amp-001}$) only concerns with integration on the set $\tg^+$!
Therefore, there is still some hope left, if the singular locus of $u$ completely misses the upper contact set.

\begin{lemma}
\label{amp-lem-001}
Let $g\in C(\bR^n)$ be a non-negative function and $u\in  C(\ol{\Omega}) \bigcap C^2(V)$.  
Suppose that $V$ is an open connected subset of $\Omega$, 
such that the upper contact set of $u$ satisfies
$$\tg^+ \subset V. $$
 Set $d:= \text{diam}(\Omega)$ and 
$$ M:= \frac{\sup_{\Omega} u - \sup_{\d \Omega}u} { d}. $$
Then we have
\begin{equation}
\label{amp-002}
\int_{B_M(0)} g(z)dV(z) \leq \int_{\tg^+} g(\nabla u(x)) |\det(D_{ij} u (x)) | dV(x).
\end{equation}
\end{lemma}
\begin{proof}
It is easy to see that the set $\tg^+$ is also relatively closed in the open subset $V$. 
Since the function $u$ is $C^2$ in an open neighbourhood of $\tg^+$, 
we can consider the mapping: 
$$\nabla u: V \rightarrow \bR^n. $$
Let the set $\Sigma$ be the image of $\tg^+$. 
Since this mapping is onto and $g\geq 0$, by change of variables, we have 
\begin{equation}
\label{amp-003}
\int_{\Sigma} g(z)dV(z) \leq \int_{\tg^+} g(\nabla u(x)) |\det (D_{ij} u(x)) | dV(x).
\end{equation}

Then it is enough to prove $B_{M}(0) \subset \Sigma$. 
In other words, we claim that for any $a\in \bR^n, |a| < M$, there exists a point $y\in \tg^+$ such that $a = \nabla u(y)$. 
Moreover, only continuity of $u$ on $\ol{\Omega}$ and $C^1$-regularity of $u$ in $V$ are needed to prove this claim.

For each such $a$, we define a linear function $L_a (t): = \min_{x\in \ol\Omega} (t + a\cdot x - u(x))$ for $t\in \bR$.
Let $t_a$ be the root of the operator $L_a$. It follows that $t_a + a\cdot x - u(x) \geq 0$ for all $x\in \Omega$,
and $t_a + a\cdot y - u(y) =0$ for some $y\in\ol\Omega$. Therefore, we have 
\begin{equation}
\label{amp-004}
u(y)\geq u(x) + a\cdot (y-x).
\end{equation}

Moreover, we can always assume that the maximum of $u$ appears in the interior, i.e. 
$u(x_0) = \sup_{\Omega} u$ for some $x_0\in \Omega$. 
Then one finds 
$$ u(y) \geq \sup_{\d\Omega} u + M\cdot d + a\cdot (y-x_0) > \sup_{\d\Omega} u.$$
Therefore, the point $y$ must be in $\Omega$, and hence it is also in the upper contact set $\tg^+$ by equation (\ref{amp-004}). 
Finally, the assumption $u\in C^2(V)$ implies that $a = \nabla u(y)$. 

\end{proof}

For an elliptic operator $L$ defined on $V$, 
we define $\mathfrak{D}^*(x), x\in V$ as the geometric average of the eigenvalue of the positive matrix $(a_{ij}(x))$. 
By picking up $g\equiv 1$, we have the following version.

\begin{corollary}
\label{amp-cor-001}
Under the condition of Lemma (\ref{amp-lem-001}), we have 
\begin{equation}
\label{amp-005}
\sup_{\Omega} u \leq \sup_{\d\Omega} u + \frac{d}{a_n} \left\{ \int_{\tg^+} \left(  \frac{- \sum_{i,j=1}^n a_{ij}(x) D_{ij} u(x)   }{ n \mathfrak{D}^*}      \right)^n dV(x)  \right\}^{1/n},
\end{equation}
where $a_n$ is the volume of the unit ball in $\bR^n$.
\end{corollary}
\begin{proof}
On the set $\tg^+$, the matrix $A = (a_{ij}(x))$ is positive, and  $D = (D_{ij}u(x))$ is semi-negative. Then we have the inequality 
$$ \mathfrak{D}^*(\det(-D))^{1/n} =  ( \det(-AD) )^{1/n}  \leq    \frac{\text{tr}(-AD)}{n}, $$
in other words, 
$$ | \det (D_{ij} u(x))|  \leq \left( \frac{ - \sum_{i,j=1}^n a_{ij}(x) D_{ij}u(x)   }{ n \mathfrak{D}^*}      \right)^n,       $$
and then our result follows.
\end{proof}

Considering the set $\Omega^+: = \{ x\in\Omega;\ u(x) >0 \}$, we further obtains the following inequality
\begin{equation}
\label{amp-006}
\sup_{\Omega} u \leq \sup_{\d\Omega} u^+ + \frac{d}{a_n} \left\{ \int_{\tg^+ \bigcap \Omega^+} 
\left(  \frac{- \sum_{i,j=1}^n a_{ij}(x) D_{ij} u(x)   }{ n \mathfrak{D}^*}      \right)^n dV(x)  \right\}^{1/n}.
\end{equation}

Up to this stage, we have seen that AMP is essentially a story of $\sup_{\Omega}u, \sup_{\d\Omega} u, \tg^+$ and the ellipticity of $L$ on $\tg^+$! 
The equation is not actually involved so far, 
and then we can formulate a new version, \emph{ Generalised Alexandrov's maximum principle} (GAMP) as follows.

\begin{theorem}
\label{amp-thm-001}
Suppose that there exists an open connected subset $V$ of $\Omega$, 
such that $\tg^+ \subset V$ and $u\in C(\ol\Omega)\bigcap C^2(V)$. 
Let $L$ be an elliptic operator with $c \leq 0$, and it satisfies $Lu \geq f$ on $V$ with 
$$ \frac{|b|}{\mathfrak{D}^*}, \frac{f}{\mathfrak{D}^*}\in L^n( \tg^+).$$
Then one has 
\begin{equation}
\label{amp-007}
\sup_{\Omega} u \leq \sup_{\d\Omega} u^+ + C\cdot \text{diam}(\Omega) \cdot \left\Vert \frac{f^-}{\mathfrak{D}^*}\right\Vert_{L^n (\Gamma^+)},
\end{equation}
with
$$  C: = C \left( n, \left\Vert \frac{|b|}{\mathfrak{D}^*} \right\Vert_{L^n(\tg^+)} \right). $$
\end{theorem}
\begin{proof}
In the case $b_i = c =0$, the proof follows directly from Corollary (\ref{amp-cor-001}) since $f\leq Lu \leq 0$ on $\tg^+$. 

For $b_i$ or $c$ non-zero, one can use Lemma (\ref{amp-lem-001}) with $V$ replaced by $V\bigcap \Omega^+$, 
and pick up $g(z):= (|z|^n + \mu^n)^{-1}$ for some constant $ \mu = || f^{-} / \mathfrak{D}^* ||_{L^n(\tg^+)}$ as the standard proof of AMP. 
\end{proof}

\section{The potential estimates}
\label{sec-003}
Let $(\vp, F)$ be a $\cC^{2,\a,\b}$-conic cscK pair on $(X,D)$. 
The coupled equations on $X\sm \text{Supp} D$ can be re-written as 
\begin{equation}
\label{pe-001}
\log\det (g_{i\bar j} + \vp_{i\bar j}) = F + \log\det (g_{i\bar j}),
\end{equation}
\begin{equation}
\label{pe-002}
\Delta_{\vp} F = - \underline{R}_\b + \text{tr}_{\vp} \Theta,
\end{equation}
where $(g_{i\bar j})$ stands for the metric for the conic K\"ahler form $\omega_\b$,
and we always assume the condition $\sup_{X} \vp =0$.  
By the construction, the function $F$ is smooth outside the divisor. 

Then we want to persuade as in Chen-Cheng \cite{CC1} to introduce an auxiliary function, by solving the following equation:
\begin{equation}
\label{pe-002}
(\omega_\b + dd^c\psi_1)^n = \frac{e^{F} \Phi(F) \omega_{\b}^n}{\int_X e^{F} \Phi(F) \omega_{\b}^n},
\end{equation}
where $\Phi(x): = \sqrt{x^2+1}$, under the normalization $\sup_X\psi_1 = 0$. 
This auxiliary function $\psi_1$ exists by a theorem of Kolodziej \cite{Kol}, and it is also $C^{\a}$-H\"older continuous on $X$.

\subsection{The auxiliary function}
In fact, we can explore more regularities of the function $\psi_1$ as in \cite{GP}. 
Take $\psi: = \psi_\b + \psi_1 $,
and equation (\ref{pe-002}) reduces to  
\begin{equation}
\label{pe-003}
(\omega + dd^c\psi)^n = \frac{e^{f} dV}{\prod_{k=1}^d |s_k|^{2-2\b_k}} ,
\end{equation}
where $f: = F + \frac{1}{2}\log (F^2 +1)$, and $dV$ is a smooth volume form. 
This is equation is very similar with the conic KE equation for $\lambda = 0$,
expect that the function $f$ may also have conic singularities.  

More precisely, the derivatives of $f$ are determined by 
$$ \d f = \d F + \frac{F\cdot\d F}{2(F^2+1)}, $$
$$ \ddbar f = \left( 1+ \frac{F}{2(F^2+1)} \right)\ddbar F - \frac{F^2}{ 2(F^2+1)^2} \d F\wedge \dbar F.$$
If $F\in\cC^{2,\a,\b}$, then there exists a constant $C$ to satisfy 
\begin{equation}
\label{pe-0030}
 |\d F|^2_ {\omega_{cone}} \leq C;
\end{equation}
\begin{equation}
\label{pe-0031}
0\leq \ddbar (F + C \sum_{k=1}^d |z_k|^{2\b_k} ) \leq 2C \omega_{cone},
\end{equation}
near the divisor. 
Moreover, the function $f$ also satisfies the above two equations by its construction. 
In fact, we have the following space.
\begin{defn}
\label{pe-defn-001}
A function $f\in C^2(X\sm \text{Supp}(D))$ is said to be in the space $\cC^{1,\bar 1}_{\b}(X,D)$, 
if equations (\ref{pe-0030}) and (\ref{pe-0031}) always hold near the support of the divisor. 
\end{defn}

The next goal is to cook up a small smooth perturbation of  the function $f$ with complex Hessian controlled near the divisor. 
Let $\{ U_i \}_{i=1}^N$ be a finite collection of open coordinate balls such that the following conditions hold: 
\begin{itemize}
\item the manifold $X$ is covered by $\bigcup_{i=1}^N U_i$;

\item there exists an integer $1\leq m < N$, such that $U_i \bigcap D_k \neq \emptyset$ for  some $k$ and $\forall i \leq m $,
and $U_i \bigcap D_k = \emptyset$ for all $k$ and $\forall i > m $.
\end{itemize}

Furthermore, for each $i\leq m$, we can assume that the defining equation of $\text{Supp}D\bigcap U_i$ is $\{z_1\cdots z_k =0 \}$,
where $\{ z_i\}$ is a coordinate system on $U_i$. 
Let $\{ \chi_i \}$ be a partition of unity subordinate to the open covering $\{ U_i\}$,
and then we can write 
$f = \sum_{j=1}^N f_j $,
where $f_j: = \chi_j\cdot f$ is compactly supported on each $U_j$. 

Let $\rho_1$ be the standard mollifier on the unit ball of $\bC^n$,
and take 
$$ \rho_{\ep} (|z|^2): = \ep^{-2n} \rho( |z|^2 / \ep^2  ).  $$ 
There exists a sequence of smooth approximation as $f_{j,\ep}: = \rho_{\ep}\star f_j$ for $j\leq m$ and $f_{j,\ep} = f_j$ for $j> m$.
In fact, we can assume the smooth function $f_{j,\ep}$ is defined on $X$ by zero extension. 
Therefore, the $f_{\ep}: = \sum_{j=1}^N f_{j,\ep}$ is a smooth function on $X$,
and converges to $f$ in $C^{\a}$-norm.
It is easy to see that the derivatives of $f_{j}$ also satisfy equations (\ref{pe-0030}) and (\ref{pe-0031}), 
and then the growth of the complex Hessian of this approximation can be estimated as follow.

\begin{lemma}
\label{pe-lem-001}
Let $B$ be the unit ball of $\bC^n$, and $\cZ$ be the zero locus of the function $\{z_1\cdots z_d \}$.
Suppose that a function $G(z)$ is in $L^1(B)$ with homogeneous growth near $\cZ$, i.e. 
there exists real numbers $\{ \a_k\}_{k=1}^d$ with each $\a_k > -1$, such that 
$$ | G(z) | \leq C \prod_{k=1}^d |z_k|^{2\a_k},  $$
in an open neighbourhood of $\cZ$. 
Then the regularization $G_{\ep} = \rho_{\ep}\star G$ has the following growth near the zero locus:
\begin{equation}
\label{pe-004}
|G_{\ep}(z)| \leq C \prod_{k=1}^d (|z_k|^2 + \ep^2)^{\a_k}.
\end{equation}
\end{lemma}
\begin{proof}
For any $x\in B_{1-\ep}$, the convolution can be estimated near $\cZ$ as 

\begin{eqnarray}
\label{pe-005}
| G_{\ep} (x) | & \leq & \int_{|z|\leq 1}  \rho(z) |G(x - \ep z) | dz
\nonumber\\
&\leq& \int_{|z|\leq 1} \rho(z) \prod_{k=1}^d |x_k - \ep z_k|^{2\a_k} dz
\nonumber\\
&\leq& \ep^{\sum_{k=1}^d 2\a_k}   \int_{|z|\leq 1} \rho(z) \prod_{k=1}^d | (x_k /\ep)- z_k|^{2\a_k} dz,
\end{eqnarray}
up to a constant. Take a function 
$$ \tilde G: =  \rho_1\star \left( \prod_{k=1}^d |z_k|^{2\a_k} \right), $$
and then we observe that this positive smooth function satisfies 
$$  \tilde G (x) \leq  M \prod_{k=1}^d (1 + |x_k|^2)^{\a_k},$$
for all $|x| < 1$ and $M = 2^d \sup_{B} \tilde G$. 
Then the last line of equation (\ref{pe-005}) can be re-written as 
\begin{eqnarray}
\label{pe-006} 
\ep^{\sum_{k=1}^d 2\a_k}  \tilde G (x / \ep) &\leq&  \ep^{\sum_{k=1}^d 2\a_k}  M \prod_{k=1}^d (1 + |x_k /\ep|^2)^{\a_k}
\nonumber\\
&\leq& M \prod_{k=1}^d (\ep^2 + |x_k |^2)^{\a_k},
\end{eqnarray}
and our result follows.
\end{proof}

The complex Hessian of each $f_{j,\ep}$ is equal to $\rho_{\ep}\star \ddbar f_{j} $.
Put $G(z) = \d_p \d_{\bar q} f_{j}$ for some $1\leq p, q \leq n $,
and then all the conditions in Lemma (\ref{pe-lem-001}) are satisfied,
where each $\a_k$ is equal to $1 -\b_k, (1-\b_k)/2$ or $0$.
Then Lemma (\ref{pe-lem-001}) shows that $dd^c f_{\ep}$ is bounded by sums of terms like 
$$ \frac{dz_p\wedge d\bar z_p}{(\ep^2 + |z_p|^2)^{\a_p}} \ \ \ \text{or}\ \ \  \frac{dz_p\wedge d\bar z_q + dz_q\wedge d\bar z_p}{(\ep^2 + |z_p|^2)^{\a'_p}(\ep^2 +|z_q|^2)^{\a'_q} }, $$
where $\a_p \in \{1-\b_p, 0 \}$ and $\a'_p\in \{ \frac{1}{2}(1-\b) , 0 \}$.

Then we can solve the following perturbed equation of equation (\ref{pe-003}):
\begin{equation}
\label{pe-007}
(\omega + dd^c \psi_{\ep})^n = \frac{e^{f_\ep} dV}{ \prod_{k=1}^d (|s_k|^2 + \ep^2)^{(1-\b_k)}},
\end{equation}
up to some uniform constant $c_{\ep}$. 
The argument for the $\cC^{1,\bar 1}$-estimate for the conic KE equation(Proposition 1, \cite{GP})
can be applied again to this equation. The only thing left is to check the following inequality: 
\begin{equation}
\label{pe-008}
dd^c f_{\ep} \geq - (C\omega_{D,\ep} + dd^c\Psi_{\ep}),
\end{equation}
for some uniform constant $C$.
Here $\Psi_{\ep}: = \sum_{k=1}^d\chi_{\rho} (|s_k|^2 + \ep^2)$ for some real number $\rho < \min_k \min\{\b_k, 1-\b_k \}$.
This is simply true by our previous estimates on the growth of $dd^c f_{\ep}$. 
Eventually, we came up with the following regularity theorem of our auxiliary function.

\begin{theorem}
\label{pe-thm-001}
The metric $\omega_{\psi_1}: = \omega_\b + dd^c\psi_1$ associated to the auxiliary function $\psi_1$
 is a conic K\"ahler metric with cone angle $\b_k$ along each divisor $D_k$.
\end{theorem} 
\begin{proof}
We already proved the $\cC^{1,\bar 1}_\b$-estimate for the potential $\psi$,
and then the metric $\omega_{\psi_1}$ is quasi-isometric to the model cone metric $\omega_{cone}$ near the divisor. 
Moreover, on each open coordinate ball $U$ with $U\bigcap \text{Supp}D = \emptyset$, the function $\psi$ is in $C^{1,\bar 1}(U)$,
and then it is in $C^{2,\a}(U)$ by the regularity result in the work \cite{DZZ}. 
Finally, the solution $\psi$ is smooth on $U$ by the standard boot-strapping technique.

\end{proof}

\subsection{$C^0$-estimate}
Suppose $\omega$ is a K\"ahler form on $X$. 
Let $\vp$ be a $\omega$-plurisubharmonic(psh) function on the manifold. 
The the regularization theorems (\cite{Dem}, \cite{BK}) of quasi-psh functions implies that 
there exists a sequence of smooth $\omega$-psh function $\vp_j$ decreasing to $\vp$. 

\begin{lemma}
\label{c0-lem-001}
There exists a real number $\a > 0$, such that for all $\omega$-psh function $\vp$ with $\sup_{X}\vp =0$, the following estimate satisfies 
\begin{equation}
\label{c0-001}
\int_X e^{-\a\vp} \omega^n \leq C_1,
\end{equation} 
for some uniform constant $C_1$ only depending on $(X,\omega)$. 
\end{lemma}
\begin{proof}
The smooth version of this lemma is established in Tian \cite{Tian}.
Taking $\vp_j$ as the smooth decreasing approximation of $\vp$, we have $\sup_X\vp_j \geq 0$, 
and then there exists two positive numbers $\a$ and $C_1$ to satisfy 
$$ \int_{X} e^{-\a \vp_j}\omega^n \leq \int_X e^{-\a(\vp_j - \sup_X\vp_j)}\omega^n \leq C_1, $$
for all $\vp$ and $j$, and our result follows. 
\end{proof}

Let $(\vp,F)$ be a $\cC^{2,\a,\b}$-conic cscK pair for the log smooth $klt$ pair $(X,D)$.
Take the normailzation $\sup_X \vp =0$, 
and then the lower bound of the potential can be first estimated in terms of $F$ and $\psi$.

\begin{theorem}
\label{main-thm-001}
Given any $\ep >0 $ small enough, there exists a constant 
$$C_2:= C_2(\ep, X, \omega, \int_X e^F\Phi(F)\omega^n, \b, \phi, D )$$ 
such that the following holds: 
\begin{equation}
\label{c0-002}
F + \ep\psi_1 - 2(1+\max_{X}|\Theta|) \vp \leq C_2.
\end{equation}
\end{theorem}

\begin{corollary}
\label{c0-cor-001}
For any real fixed number $q>0$, there exists a constant 
$$C_3: = C_3(q, X,  \omega, \int_X e^F\Phi(F)\omega^n, \b, \phi, D, \max_X|\Theta| ) $$
such that the following holds:
\begin{equation}
\label{c0-003}
\int_X e^{qF} \omega^n \leq C_3;\ \ \ || \vp ||_0 \leq C_3; \ \ \ ||\psi_1||\leq C_3.
\end{equation}
\end{corollary}
\begin{proof}
Combing with Lemma (\ref{c0-lem-001}) and Theorem (\ref{main-thm-001}),
the uniform $L^q(\omega^n)$ estimate for the function $e^F$ follows exactly like the argument in Chen-Cheng \cite{CC1}, by picking up $\ep = \a/q$.
However, in order to prove the $L^{\infty}$ bound of the potential, we need the following argument. 

Re-write the Monge-Amp\`ere equation (\ref{pe-001}) as 
$$ (\omega + dd^c (\psi_\b + \vp))^n = \mu \omega^n,$$
where $\mu: =  \frac{e^F\omega^n_\b}{\omega^n}$ is the density function. 
Taking some $1< p < \min_k (1- \b_k)^{-\frac{1}{2}}$, the $L^p$-norm of $\mu$ can be estimated by the H\"older inequality 
\begin{eqnarray}
\label{c0-004}
\int_X e^{pF}  \frac{dV}{\prod_{k=1}^d |s_k|^{2p(1-\b_k)}} &\leq&
\left( \int_X e^{ (p+p') F} dV \right)^{1/p'} \left( \int_X  \frac{dV}{\prod_{k=1}^d |s_k|^{2p^2(1-\b_k)}}  \right)^{1/p}
\nonumber\\
&\leq& C_4(\b) \left( \int_X e^{ q F} dV \right)^{1/p'},
\end{eqnarray}
where we choose $\frac{1}{p} + \frac{1}{p'}=1$ and $q: = p + p'$. 
Finally, by the work of Kolodziej \cite{Kol} and Benelkourchi-Guedj-Zeriahi \cite{BGZ},
the $L^{\infty}$-norm is controlled as 
$$ 0\leq ||\vp + \psi_\b ||_0 \leq C_5 || \mu ||^{1/n}_{L^p (\omega^n)}, $$
where the constant $C_5$ only depends on $p$ and $\omega$.
Then our result follows since $\psi_\b$ is uniformly bounded.
Moreover, for the auxiliary function $\psi_1$, we have the same $L^{\infty}$-estimate,
since $\sqrt{F^2+1}$ is controlled by $e^{\ep_1 F}$ for any small $\ep_1 >0$. 

\end{proof}

For the proof of Theorem (\ref{main-thm-001}), we first run as Chen-Cheng's argument \cite{CC1}.
Take $u_1: = e^{\delta A_1} $ and $A_1(\vp, F): = F + \ep\psi_1 - \lambda \vp$.
Here the constants are determined as 
$$ \lambda: = 2(1+ \max_X |\Theta|),\ \ \  \delta: = \frac{\a}{2n\lambda},$$
where $\a$ is the small constant appearing in Lemma (\ref{c0-lem-001}).

Let $p_0$ be the maximum point of the function $u$.
For some $d>0$ small enough, we can consider a coordinate ball $B_{d}(p_0)$ around $p_0$ with radius $d$.
Take $\eta_{p}$ be a cut-off function such that $\eta_p(p)=1$ and $\eta_p = 1-\theta$ outside the ball $B_{d/2}(p)$,
with the estimate $|\nabla \eta_p|^2 \leq 4\theta^2 d^{-2}$ and $|\nabla^2\eta_p|\leq 4\theta d^{-2}$.
This small positive constant $\theta$ will only depend on $\a$ and $d$.

Suppose that the ball $B_{p_0}(d)$ is away from $\text{Supp}(D)$.
Then the estimate (\ref{c0-002}) holds for some constant, by applying AMP to the function $u\eta_{p_0}$.
However, this constant will depend on the diameter $d$, and it grows like $d^{-1}$ when the ball is closer and closer to the divisor.

Therefore, we need to introduce a new auxiliary function as 
$$\psi_2: = \sum_{k=1}^d |s_k|^{2\g_k}.$$


Put $u: = e^{\delta A}$, where we define 
$$A(\vp, F): = F + \ep (\psi_1 + \psi_2) -\lambda \vp. $$
Let $x_0$ be the maximum point of the function $u$ on the manifold,
and we can assume $B_{d}(x_0)\bigcap \text{Supp}(D)\neq \emptyset$ for some fixed radius $d$ small. 
Then there exists an open coordinate system $U$ such that $B_{2d}(x_0)\subset U $,
and the defining function of $\text{Supp}(D)$ is $\{z_1\cdots z_d =0 \}$ in $U$. 

We re-wirte the new auxiliary function $\psi_2$ on $U$ as  
$\sum_{k=1}^d  ( |z_k|^{2}e^{-\phi_k} )^{\g_k}$, 
and its complex Hessian $dd^c\psi_2$ can be explicitly calculated as 
\begin{equation}
\label{c0-006}
\sum_{p=1}^d \g_p^2 e^{-\phi_p} \frac{  dz_p\wedge d\bar z_p +   2\Re \left\{ \sum_{q=1}^n  o(z_p) dz_q\wedge d\bar z_p \right\} + \sum_{q,l=1}^n o( |z_p|^2)  dz_q\wedge d\bar z_l           } {(  |z_p|^2e^{-\phi_p} )^{1-\g_p} }.
\end{equation}
Therefore, we have the following estimate near the divisor 
\begin{equation}
\label{c0-007}
C_6 \omega_{Euc}+ dd^c\psi_2 \geq C^{-1}_6 \left( \sum_{k=1}^d \frac{dz_k \wedge d\bar z_k}{|z_k|^{2-2\g_k}} + \sum_{j = d+1}^n dz_j\wedge d\bar z_j \right),
\end{equation}
for some constant $C_6$ only depending on the angle $\g$ and the hermitian metric $\phi$.

The complex Hessian function of $\psi_2$ grows very fast to $+\infty$ near the divisor, 
and this gives us a chance to avoid its upper contact set. 

\begin{lemma}
\label{c0-lem-002}
Let $\tg^+$ be the upper contact set of the function $u\eta_{x_0}$ on $U$.
Then there exists an open neighbourhood $V_D$ of $\text{Supp}(D)\bigcap U$, such that $\tg^+ \bigcap V_D =\emptyset$. 
\end{lemma}
\begin{proof}
By our construction, the function $u\eta_{p_0}$ is smooth outside the divisor, 
and then we can compute its Laplacian with respect to the Euclidean metric $\omega_{Euc}$ on $U\sm \text{Supp} D$ as
\begin{eqnarray}
\label{c0-008}
\Delta (e^{\delta A}\eta) &=& \Delta(e^{\delta A})\eta + e^{\delta A} \Delta\eta + 2\delta e^{\delta A} \nabla A \cdot \nabla \eta
\nonumber\\
&=& e^{\delta A}\eta \left( \delta \Delta A + \delta^2 |\nabla A|^2     \right) + e^{\delta A} \Delta\eta + 2\delta e^{\delta A} \nabla A \cdot \nabla \eta.
\end{eqnarray}
By the construction, we have 
\begin{equation}
\label{c0-009}
e^{\delta A} \Delta\eta \geq - e^{\delta A} 4\theta /d^2,
\end{equation}
and 
\begin{eqnarray}
\label{c0-010}
2\delta \nabla A \cdot \nabla \eta &\geq& -\delta^2\eta |\nabla A|^2 -  \eta^{-1} |\nabla\eta|^2
\nonumber\\
&\geq&  -\delta^2\eta |\nabla A|^2 -  \frac{4\theta^2}{d^2(1-\theta)}.
\end{eqnarray}
Moreover, since $\vp, F, \psi_\b\in\cC^{2,\a,\b}$, we see
\begin{eqnarray}
\label{c0-011}
\Delta (F + \ep\psi_1 + \ep \psi_2 - \lambda \vp) &\geq& \text{tr}_{\omega_{Euc}} \{  dd^c(F - \lambda\vp) - \ep\omega_{\b} \} + \ep \Delta \psi_2
\nonumber\\
&\geq&  - C_7 \left(\sum_{k=1}^d |z_k|^{2\b_k -2} + 1 \right) + \ep C^{-1}_7 \left( \sum_{k=1}^d |z_k|^{2\g_k -2}  \right),
\end{eqnarray}
for some constant $C_7$(may not be uniform). 
Eventually, for chosen $\ep, \delta, \lambda$ and $\theta$, we have on $U\sm \text{Supp}(D)$
\begin{eqnarray}
\label{c0-012}
\Delta (e^{\delta A}\eta) &\geq& \delta \eta e^{\delta A} \Delta (F + \ep\psi_1 + \ep \psi_2 - \lambda \vp) -  e^{\delta A}  \left( \frac{4\theta^2}{d^2(1-\theta)} + \frac{4\theta}{d^2}   \right)
\nonumber\\
&\geq&  - C_8 \left(\sum_{k=1}^d |z_k|^{2\b_k -2} + 1 \right) + C^{-1}_8 \left( \sum_{k=1}^d |z_k|^{2\g_k -2}  \right),
\end{eqnarray}
for some constant $C_8$. 

By picking up $\g_k < \b_k$,
there exists an open neighbourhood $V_D$ of the support of the divisor $D$ in $U$ such that 
$ \Delta (u\eta_{x_0}) > 1$ on $V_D\sm\text{Supp}(D)$. Therefore, the upper contact set $\tg^+$ is disjoint from the open set $V_D\sm\text{Supp}(D)$. 

Furthermore, we claim that $\text{Supp}(D)\bigcap \tg^+ =\emptyset$. 
Otherwise, suppose a point $p_0$ is in $ \text{Supp}(D)\bigcap \tg^+$,
and then there exists a vector $a\in \bR^{2n}$, such that 
$$ u\eta_{x_0}(y) \leq u\eta_{x_0}(p_0) + a\cdot (y-p_0),$$
for all $y\in U$. Define a new function on $V_D$ as 
$$ v(y): =  u\eta_{x_0}(y) + a\cdot (p_0 -y).$$
By our construction, this function $v$ obtains its maximum at the point $p_0$, and it is continuous on $V_D$.
Moreover, the function $v$ is strictly subharmonic on $V_D\sm \text{Supp}(D)$ since its real Laplacian $\Delta v = \Delta (u\eta_{x_0})$ is positive there.

In fact, $v$ is even subharmonic on the whole $V_D$ by the extension theorem of subharmonic functions. 
Thanks to the maximum principle, it must be a constant in $V_D$, but this contradicts the fact that $v$ is strictly subharmonic outside the divisor.

\end{proof}

\begin{proof}[Proof of Theorem (\ref{main-thm-001})]
In order to apply the maximum principle to our function $u\eta_{x_0}$, we compute the Laplacian $\Delta_{\vp} (u\eta_{x_0})$ outside the divisor.
The calculation is very similar with Chen-Cheng's work \cite{CC1}, and 
the only difference is that our background metric is $\omega_\b$, a conic metric this time. However, observe that we have 
$ Ric(\omega_\b) = \Theta$
outside the divisor. Therefore, we obtain 
\begin{equation}
\label{c0-1200}
\begin{split}
& \Delta_\vp (F + \ep\psi_1 + \ep\psi_2 - \lambda\vp) \geq \left\{ - \ul{R}_\b - \lambda n + \ep n I_{\Phi}^{-\frac{1}{n}} \Phi^{\frac{1}{n}}(F)  \right\} 
\\
& + \ep\Delta\psi_2 + (\lambda - \ep - |\Theta| ) \tr_\vp g_\b,
\end{split}
\end{equation}
where $I_{\Phi}: = \int_X e^F\Phi(F)\omega_{\b}^n$.
Here we used the following inequality 
$$ \Delta_\vp \psi_1 \geq n\left( e^{-F} e^F \Phi(F) I_{\Phi}^{-1} \right)^{\frac{1}{n}} - \tr_\vp g_\b.$$

Moreover, for some large integer $N$, we have the Donaldson metric
$$ \omega_{D,\g} = \omega + N^{-1}dd^c\psi_2 >0,$$
with cone angle $\g_k$ along $D_k$. Therefore, we see 
\begin{equation}
\label{c0-0115}
\ep \Delta_{\vp}\psi_2 \geq -\ep N \text{tr}_{\omega_{\vp}} \omega \geq  -\ep N_1 \text{tr}_{\omega_{\vp}} \omega_\b.
\end{equation}
Here the constants $N$ and $N_1$ only depend on $\omega, \omega_{\b}, X, D$ and the hermitian metric $\phi$.
Then we may assume $\ep N_1 < 1$, and the following inequality holds: 
\begin{equation}
\label{label_01201}
\begin{split}
& \Delta_\vp \left( e^{\delta A}\eta_{x_0}   \right) \geq \delta\eta_{x_0} e^{\delta A} \left(- \underline{R}_\b - \lambda n + \ep n I_{\Phi}^{-\frac{1}{n}} \Phi^{\frac{1}{n}}(F) \right)
\\
& + e^{\delta A} \left( \delta\eta_{x_0} (\lambda -\ep - |\Theta| -1) - \frac{4\theta}{d^2} - \frac{4\theta^2}{d^2(1-\theta)}  \right)\tr_\vp g_\b.
\end{split}
\end{equation}
Recall these constants are taken as  
$ \lambda: = 2(1+ \max_X |\Theta|)$, and $\delta: = (2n\lambda)^{-1}\a$,
and then we choose the constant $\theta >0$ small enough to satisfy 
$$ \frac{(1-\theta)\a}{4n} - \frac{4\theta}{d^2} - \frac{4\theta^2}{d^2(1-\theta)} \geq 0.$$
This implies the following equation on $U\sm \text{Supp}(D)$
\begin{equation}
\label{c0-014}
\Delta_{\vp} (u\eta_{x_0}) \geq \delta\eta_{x_0} e^{\delta A} (- \underline{R}_\b - \lambda n + \ep n I_{\Phi}^{-\frac{1}{n}} \Phi^{\frac{1}{n}}(F)).
\end{equation}

Thanks to Lemma (\ref{c0-lem-002}), the upper contact set $\tg^+$ of the continuous function $u\eta_{x_0}$ in the ball $B_d(x_0)$
is contained in the open subset $B_d(x_0)\sm \ol{V_D}$, which is away from the divisor. 
Then we are ready to apply GAMP in the ball to have: 
\begin{eqnarray}
\label{c0-015}
&&\sup_{B_d(x_0)} u\eta_{x_0} \leq \sup_{\d B_d(x_0)} u\eta_{x_0} 
\nonumber\\
&+& C_n d_0 
\left( \int_{B_d(x_0)\bigcap \Omega^- } e^{2F} u^{2n}  (- \underline{R}_\b - \lambda n + \ep n I_{\Phi}^{-\frac{1}{n}} \Phi^{\frac{1}{n}})^{2n} \omega^n  \right)^{\frac{1}{2n}},
\nonumber\\
\end{eqnarray}
where $\Omega^-$ denote the set 
$$ \Omega^-: = \{ x\in B_d(x_0); \ \  - \underline{R}_\b - \lambda n + \ep n I_{\Phi}^{-\frac{1}{n}} \Phi^{\frac{1}{n}}< 0  \}. $$
Then there exists a constant $C_9$ only depending on $\ep, I_{\Phi}$ and the smooth metric $\omega$ such that $F < C_9$ on $\Omega^-$.
As in Chen-Cheng \cite{CC1}, the last term is eventually bounded by the integral 
\begin{eqnarray}
\label{c0-016}
(|\underline{R}_\b| +\lambda n)^{2n} e^{C_9(2n\delta +2)} \int_{B_d(x_0)} e^{-\a\vp} \omega^n \leq C_{10},
\end{eqnarray}
by Lemma (\ref{c0-lem-001}). Therefore, we obtain 
$$ \sup_X u = u\eta (x_0) \leq (1-\theta) \sup_X u + C_nd\cdot C_{10}, $$
and then $\sup_X u \leq \theta^{-1} C_nd\cdot C_{10}$.
Finally our result follows since the function $\psi_2$ is uniformly bounded on $X$.

\end{proof}

Eventually, the $C^0$-norm of the potential $|| \vp ||_0$ is controlled by the conic entropy $H_\b(\vp): = \int_X F e^F \omega_\b^n$,
by the equivalence between the integral $\int_X e^F \sqrt{F^2+1}\omega_\b^n$ and $H_\b$ as in \cite{CC1}.

\subsection{Non-degeneracy estimate}
The uniform upper bound of the volume form radio $F: = \log\frac{\omega^n_{\vp}}{\omega^n_{\b}}$ is easily obtained from the inequality (\ref{c0-002}).
In fact, for a fixed $\ep_0$ small enough, we have 
$$F \leq C_2 -\ep_0 \psi_1,$$
and the result follows from Corollary (\ref{c0-cor-001}).

The last issue is the lower bound of $F$, but we can use GAMP again as follows.
\begin{lemma}
\label{nd-lem-001}
There exists a constant $C_{11}$ satisfying 
\begin{equation}
\label{nd-001}
F \geq - C_{11},
\end{equation}
where the constant depends on 
$$C_{11}: = C_{11}( ||\vp||_0, X,  \omega, \b, \phi, D, \max_X|\Theta| ). $$ 
\end{lemma}
\begin{proof}
As before, we consider a function 
$$ A_2(F, \vp):= - F - \lambda\vp + \ep_2\psi_2, $$
and put $u_2 = e^{\delta A_2}$. 
Assume the function $u_2$ achieves its maximum at the point $x_1$.
Pick up 
$$\lambda: = 2(\max_X|\Theta| +1);\ \ \  \delta:= \frac{1}{2n}; \ \ \  \ep_2 = \frac{1}{2N_1}, $$
where $N_1$ is the uniform constant in equation (\ref{c0-0115}).
Then we have 
\begin{eqnarray}
\label{nd-002}
\Delta_{\vp} (F+\lambda \vp - \ep_2\psi_2 ) &=& \left( \tr_{\vp}\Theta - \underline{R}_{\b}  - \lambda \text{tr}_{\vp}\omega_{\b} \right) + \lambda n - \ep_2\Delta_{\vp}\psi_2
\nonumber\\
&\leq& - \text{tr}_{\vp}\omega_\b + \underline{R}_{\b} + \lambda n + \ep_2N_1 \text{tr}_{\vp}\omega_{\b},
\nonumber\\
&\leq& - \frac{1}{2}\text{tr}_{\vp}\omega_\b + \underline{R}_{\b} + \lambda n 
\end{eqnarray}
outside the support of the divisor. 
Following Chen-Cheng's calculation, we further see
\begin{equation}
\label{nd-002}
\Delta_{\vp} (u_2\eta_{x_1}) \geq  \delta e^{\delta u_2}\left\{ \text{tr}_{\vp} g \left(\frac{1}{2}\delta\eta_{x_1} - \frac{2\theta}{d^2} -\frac{4\theta^2}{d^2(1-\theta)} \right)    
- \delta(\underline{R}_{\b} + \lambda n) \right\}.
\end{equation}
Choose $\theta$ sufficiently small to satisfy 
$$\frac{1-\theta}{2}\delta- \frac{2\theta}{d^2} -\frac{4\theta^2}{d^2(1-\theta)}  \geq 0, $$
and then we have 
\begin{equation}
\label{nd-003}
\Delta_{\vp} (u_2\eta_{x_1}) \geq -\delta e^{\delta A_2}( \ul{R}_{\b} + \lambda n). 
\end{equation}

Now observe that the function $u_2\eta_{x_1}$ is strictly subharmonic in an open neighbourhood of the divisor,
by the same argument as in Lemma (\ref{c0-lem-002}).
Then there exists an open subset $V$ of the ball $B_{d}(x_1)$ completely disjoint from the divisor,
such that the upper contact set $\tg^+_{(u_2\eta_{x_1})}$ of the function $u_2\eta_{x_1}$ is contained in $V$.
Therefore, we can apply GAMP to this function on the ball $B_d(x_1)$
\begin{eqnarray}
\label{nd-004}
&&e^{\delta A_2}\eta_{x_1}(x_1) \leq \sup_{\d B_d(x_1)} e^{\delta A_2}\eta_{x_1}
\nonumber\\
&+& C_n d \left( \int_X e^{2F} e^{-2n\delta A_2} (\ul{R}_{\b} + \lambda n)^{2n}\omega^n \right)^{\frac{1}{2n}}.
\end{eqnarray}
However, this integral is bounded by the following 
\begin{equation}
\label{nd-005}
 \int_X e^{2F} e^{-2n\delta A_2} (\ul{R}_{\b} + \lambda n)^{2n}\omega^n \leq C_{12} \int_X e^{(2-2n\delta) F} \omega^n \leq C_{12}\int_X e^F \omega^n,
\end{equation}
and our result follows.

\end{proof}

\begin{rem}
\label{c0-rem-001}
During the proof of the a priori estimates, the regularity condition $(\vp, F)\in \cC^{2,\a,\b}$ is more than enough.
In fact, we can prove our results by only assuming $(\vp, F)\in \cC^{1,\bar 1}_{\b}$. 
\end{rem}

\begin{rem}
\label{c0-rem-002}
The constant $C_2$ and $C_3$ depend on many things as listed before, but they do not actually depend on the conic background metric $\omega_\b$.
In other words, if we switch our background metric to another conic metric $\tilde\omega_{\b}$ which is isometric to $\omega_\b$,
then the uniform estimate also works, 
with $\max_X|\Theta|$ replaced by $\max_{X\sm \text{Supp}(D)} | Ric(\tilde\omega_\b)|$.
\end{rem}

\section{The $W^{2,p}$ estimates}
\label{sec-004}
In this section, we want to demonstrate the estimates on the Laplacian of the potential $\vp$ for conic cscK equations.
Taking $Y: = X\sm \text{Supp}(D)$,
the idea is to first prove the $W^{2,p}(d\mu, Y)$ estimate for $\vp$ for some measure $d\mu$,
and then use the $W^{2,p}(d\mu, Y)$ norm to control the $L^{\infty}$-norm of the Laplacian. 

\begin{theorem}
\label{main-thm-2}
For any $p \geq 1$, there exists a constant $C_{14}$ satisfying 
\begin{equation}
\label{main2}
\int_{Y} ( \text{tr}_{\omega_\b} \omega_{\vp} )^{p} \omega_{\b}^n \leq C_{14},
\end{equation}
where this constant depends on
$$ C_{14}: = C_{14} ( \ p, ||\vp||_0, || F ||_0,  \omega_D, \omega_\b, X, D).$$
\end{theorem}

In Chen-Cheng's proof \cite{CC3} of the $W^{2,p}$ estimate, this constant $C_{14}$ actually is related to the lower bound of the bisectional curvature $R_{i\bar i j\bar j}$ 
of the background metric. However, the background metric $\omega_\b$ in our case is singular, 
and the growth of its bisectional curvature near the divisor is not clear up to now.
Therefore, we need to switch our background metric back to Donaldson's metric in this section, as in Guenancia-P\u aun \cite{GP}.
In fact, since the two conic metrics $\omega_{\b}$ and $\omega_{D}$ are quasi-isometric on $X$, it is enough to prove the following 
\begin{equation}
\label{w2p-001}
\int_{Y} ( \text{tr}_{\omega_D} \omega_{\vp} )^{p} \omega_{D}^n \leq C,
\end{equation}
for some uniform constant $C$(may depends on $p$). 

\subsection{Conic weight function}
Let $\Psi_{\g}$ be an auxiliary function defined on $X$ as 
$$ \Psi_{\g}:= C\sum_{k=1}^d |s_k|^{2\g}, $$
for some $\g < \min_k \min\{  \b_k, 1-\b_k \}$.
Then it is the smooth limit on $Y$ of the auxiliary function 
$$\Psi_{\g, \ep}: = C \sum_{k=1}^d \chi_{\g}(\ep^2 + |s_k|^2),$$
constructed in \cite{GP}.
Then the function $\Psi_\g$ is clearly $C_{15}\omega_D$-$psh$ for another uniform constant $C_{15}$, i.e. we have
\begin{equation}
\label{w2p-0015}
C_{15}\omega_D + dd^c \Psi_\g \geq 0,
\end{equation}
on $X$.
Let $\Theta_{\omega}(T_X)$ denote the Chern curvature tensor of $(T_X, \omega)$.
The following inequality is proved in \cite{GP}:
$$ \sqrt{-1}\Theta_{\omega_{D,\ep} }(T_X) \geq - (C_{16}\omega_{D,\ep} + dd^c \Psi_{\g, \ep})\otimes \text{Id}.$$
In a normal coordinate of the metric $\omega_{D,\ep}$, we can re-write the above inequality as 
$$ R_{i\bar i j\bar j}(\omega_{D,\ep}) \geq -(C_{16} + \Psi_{\ep, i\bar i} ); \ \ \ \text{and} \ \ \   R_{i\bar i j\bar j}(\omega_{D,\ep}) \geq -(C_{16} + \Psi_{\ep, j\bar j} ).$$
Therefore, the following holds on $Y$ in a normal coordinate of $\omega_D$:
\begin{equation}
\label{w2p-002}
R_{i\bar i j\bar j}(\omega_{D}) \geq -(C_{16} + \Psi_{ i\bar i} ); \ \ \ \text{and} \ \ \   R_{i\bar i j\bar j}(\omega_{D}) \geq -(C_{16} + \Psi_{ j\bar j} ),
\end{equation}
since everything converges smoothly outside the divisor. 
Moreover, if we put 
$$h_{\ep}:= -\log\left(  \frac{\prod_{k=1}^d (\ep^2 + |s_k|^2)^{1-\b_k} \omega^n_{D,\ep}}{dV} \right)$$
for some smooth volume form $dV$,
then the following also holds by the calculation in \cite{GP}
$$ C_{17}\omega_{D,\ep} + dd^c\Psi_{\g,\ep}  \geq dd^c h_{\ep} \geq -(C_{17}\omega_{D,\ep} + dd^c\Psi_{\g,\ep}),$$
for some uniform constant $C_{17}$. Taking the limit, we have the following estimate on $Y$
\begin{equation}
\label{w2p-003}
C_{17}\omega_D + dd^c \Psi_\g   \geq    dd^c h \geq - (C_{17}\omega_D + dd^c \Psi_\g),
\end{equation}
where the function
$h: = \log\left( \frac{\omega^n_{\b}}{\omega^n_D}\right) $
is defined on $Y$.
In fact, a direct computation shows  
\begin{equation}
\label{w2p-0030}
\ddbar |s_k|_{\phi_k}^{2-2\b_k} = \b_k^2 \frac{\d^{\phi_k} s_k \wedge\ol{\d^{\phi_k}s_k }e^{-\phi_k}  }{|s_k|_{\phi_k}^{2-2\b_k}} - \b_k |s_k|^{2\b_k}\ddbar\phi_k,
\end{equation}
and then one obtains 
\begin{equation}
\label{w2p0035}
e^{-h} dV = \sum_L \sum_{I,J} \left( \prod_{i\in I} |s_i|^{2-2\b_i}_{\phi_j} \right) \left( \prod_{j\in J} |\d^{\phi_j} s_j |^2_{\phi_j}  \right) 
\wedge\left(  \prod_{l\in L} |s_l |_{\phi_l}^{2\b_l} \ddbar\phi_l \right) \varrho,
\end{equation}
where $\{ I, J \}$ is any partition of the set $\{1,\cdots, d \}$,
$L$ is a subset of $\{ 1,\cdots, n \}$ with possibly repeating indices, 
and $\varrho$ is a smooth function.
At a point $p$ near the divisor, we can assume $\d\phi_k (p) = 0$ for all $1\leq k\leq d$.
Therefore, the growth of its gradient can be computed as  
$$ \d_k h =  O ( |z_k|^{-\max\{1-2\b_k, 2\b_k -1\} } )  $$
for $1\leq k \leq d$, and $\d_p h = O(1)$ for $ d < p \leq n$. 
Moreover, its complex Hessian can be estimated as 
$$ \d_{p}\d_{\bar p} h = O (|z_p|^{2\a_p});\ \ \ \d_p\d_{\bar q} h = O(|z_p|^{\a'_p} |z_q|^{\a'_q}),$$
where $\a_p \in \{1-\b_p, \b_p \}$ and $\a'_p \in \{ \frac{1}{2}-\b_p, \b_p - \frac{1}{2} \}$ for all $1\leq p, q \leq d$.

\subsection{Switching background metrics}
In the following, we will slightly change our notations.
Let $(\psi, G)$ be a $\cC^{2,\a,\b}$-conic cscK pair for $(X,D)$, i.e. they satisfy the coupled equations (\ref{pe-001}) and (\ref{pe-002}).

Take a new potential $\vp: = \psi + \psi_\b -\psi_D $,
and a new function $F: = G + h$.
The the new potential $\vp$ is also in $\cC^{2,\a,\b}$, and it satisfies 
$$ \omega_{\vp}: = \omega_D + dd^c \vp = \omega_\b + dd^c\psi, $$
and the function $F$ is uniformly bounded, but it may not be in $\cC^{2,\a,\b}$ anymore.
The two coupled cscK equations (\ref{csck-002}), (\ref{csck-003}) can be re-written as 
\begin{equation}
\label{w2p-004}
(\omega_D + dd^c \vp)^n = e^F \omega_D^n;
\end{equation}
\begin{equation}
\label{w2p-005}
\Delta_{\vp} F = \text{tr}_{\vp}(\Theta + dd^c h) - \ul{R}_{\b}.
\end{equation}

In fact, the $(1,1)$ closed form $(\Theta + dd^c h)|_{Y}$ is the restriction of the curvature $Ric(\omega_D)$ on $Y$.
From now on, we will adapt to the following conventions:
\begin{itemize}

\item denote $g, \nabla$ and $\Delta$ with respect to the background metric $\omega_D$;

\item denote $g_{\vp}, \nabla_{\vp}$ and $\Delta_{\vp}$ with respect to the target metric $\omega_{\vp}$.
\end{itemize}

In order to manipulate the integration by parts on $Y$, we need to introduce a suitable cut off function.
Let $\rho: X \rightarrow [-\infty, + \infty] $ be a function defined by
$$\rho (x):= \log\left( -\log(\prod_{k=1}^d |s_k(x)|^2 ) \right), $$
where we normalise the sections $\tau: = \prod_{k=1}^d |s_k(x)|^2 < e^{-1}$.
Let  $\eta_{\ep}: [0, +\infty ) \rightarrow [0,1]$ be a smooth non-decreasing function,
such that $\eta (x) =0$ for $x\in [0,1]$ and $\eta(x) = 1$ for $x\geq 2$. 
Then the following cut off function is considered in Berndtsson's work \cite{Bo1}
$$ \theta_\ep(x): = 1 - \eta(\ep\rho(x)),$$
and it is equal to $1$ whenever $\tau \geq e^{-e^{1/\ep}}$ and $0$ if $\tau \leq e^{-e^{2/\ep}}$.

Moreover, its gradient is 
\begin{equation}
\label{w2p-0060}
\d \theta_{\ep} =  \frac{\ep\eta'}{-\log\tau} \sum_{k=1}^d \left( \frac{\d^{\phi_k } s_k}{s_k} \right).
\end{equation}
The positive $(1,1)$ form $\d\theta_\ep\wedge \dbar \theta_\ep$ is only supported near the divisor, 
and we have its integrability with respect to the model cone metric
\begin{equation}
\label{w2p-0061}
\int_X  \frac{1}{(\log\tau)^2} \sum_{k, l=1}^d \left( \frac{\d^{\phi_k } s_k}{s_k} \right)\wedge \ol{ \left( \frac{\d^{\phi_l } s_l}{s_l} \right)}\wedge \omega_{cone}^{n-1 }< +\infty.
\end{equation}
Then it is easy to see that the following property  holds:
\begin{equation}
\label{w2p-007}
\int_{X} d\theta_\ep \wedge d^c \theta_\ep \wedge \Omega_\b^{n-1} \rightarrow 0,\ \ \ \text{as}\ \  \ep\rightarrow 0
\end{equation}
for any conic K\"ahler metric $\Omega_\b$ on $X$. 
Moreover, this implies the following 
\begin{equation}
\label{w2p-0070}
\int_X  d \theta_\ep \wedge d^c F \wedge  \Omega_\b^{n-1} \rightarrow 0,\ \ \ \text{as}\ \  \ep\rightarrow 0,
\end{equation}
for any function $F$ with $|\d F |^2_{\Omega_\b} \in L^2( \Omega_{\b}^n) $.
These properties will be crucial for our later calculation.

\begin{proof}[Proof of Theorem (\ref{main-thm-2})]
As discussed before, it is enough to prove the $W^{2,p}$ estimate 
with respect to the new background metric $\omega_D$
(equation (\ref{w2p-001})).
Let $\k>0, C>0$ be constants to be determined later. 
According to Guenancia-P\u aun's trick, the following Laplacian can be estimated on $Y$ as 
\begin{equation}
\label{w2p-006}
\Delta_{\vp} \log (n + \Delta\vp)     \geq  - C_{16} \text{tr}_{\vp} g + \frac{\Delta F}{ n + \Delta\vp} - \Delta_\vp \Psi,
\end{equation}
where $\Psi: = \Psi_\g$ is the conic weight function.
Denote a function $A(\vp, F): = -\k(F + C\vp) + (\k+1)\Psi$, and then compute 
\begin{eqnarray}
\label{w2p-006}
&& e^{-A}\Delta_\vp (e^A (n + \Delta\vp)) \geq  (n+ \Delta\vp) \Delta_{\vp} (A + \log(n+ \Delta\vp))
\nonumber\\
&\geq&  (n+ \Delta\vp) \big\{ (\k C - C_{16}) \tr_{\vp}g   - \k \tr_\vp\Theta + \k (\ul{R}_{\b} - Cn) 
\nonumber\\
&+& \k\cdot \tr_\vp (dd^c\Psi - dd^c h) + (n+ \Delta\vp)^{-1}\Delta F \big\}
\nonumber\\
&\geq& \frac{\k C}{4} \tr_\vp g (n+\Delta\vp) + \Delta F - \k C_{18}(n + \Delta\vp),
\end{eqnarray}
where we choose the constant $\k \geq 1$ and 
$$C: = 4(\max|\Theta|_g + C_{17} + C_{16} +1).$$
Here we used equations (\ref{w2p-002}) and (\ref{w2p-003}).

Let $p > 1$, and $0 < \delta < (p-1)/10$,  denote $v:= e^{A}(n+ \Delta\vp)$ as an $L^{\infty}$ function on $X$, and then we have
\begin{eqnarray}
\label{w2p-008}
&&(p-1)\int_X \theta_\ep^2 v^{p-2}  |\nabla_{\vp} v|^2_{\vp} \omega^n_\vp 
\nonumber\\
&=& \int_X \theta_\ep^2 v^{p-1}(-\Delta_\vp v)\omega_\vp^n 
- 2\int_X (v \nabla_\vp\theta_\ep ) \cdot_{\vp}  ( \theta_\ep\nabla_{\vp}v) v^{p-2}\omega^n_{\vp}
\nonumber\\
&\leq& \int_X \theta_\ep^2 v^{p-1}(-\Delta_\vp v)\omega_\vp^n  +  \delta \int_X \theta_\ep^2 v^{p-2} |\nabla_{\vp} v|^2_{\vp} \omega_{\vp}^n + \delta^{-1} {\rm{I}}_{\ep},
\end{eqnarray}
where the last term is 
$${\rm{I}}_{\ep}: = \int_X |\nabla_\vp \theta_\ep|^2_{\vp} v^{p}\omega_{\vp}^n = \int_X v^p d\theta_\ep\wedge d^c\theta_\ep \wedge \omega^{n-1}_{\vp}.$$
Moreover, we have 
\begin{eqnarray}
\label{w2p-009}
&&(p-1 -\delta)\int_X  \theta_\ep^2 v^{p-2}  |\nabla_{\vp} v|^2_{\vp} \omega^n_\vp 
\nonumber\\
&\leq& \delta^{-1} {\rm I}_\ep - \int_X \theta_\ep^2 v^{p-1} \left(\frac{\k C}{4} v \tr_\vp g + e^A \Delta F - \k C_{18} v \right).
\end{eqnarray}

We will handle the term involving $\Delta F$ as in Chen-Cheng \cite{CC3}
\begin{eqnarray}
\label{w2p-010}
&& - \int_X \theta_\ep^2 v^{p-1} e^{A} \Delta F \omega^n_\vp = -\int_X \theta_\ep^2 e^{(1-\k )F - \k C\vp + (\k +1)\Psi } \Delta F \omega_D^n
\nonumber\\
&=& - \int_X \theta_\ep^2 v^{p-1} e^{(1-\k )F - \k C\vp + (\k +1)\Psi } \frac{1}{1-\k} \Delta \big( (1-\k)F - \k C\vp + (1+\k)\Psi  \big)\omega_D^n
\nonumber\\
&-&   \int_X \theta_\ep^2 v^{p-1}e^{(1-\k )F - \k C\vp + (\k +1)\Psi }  \frac{  \k C \Delta\vp -(1+\k)\Delta\Psi}{1-\k}\omega_D^n.
\nonumber\\
\end{eqnarray}
Put $B(\vp, F, \Psi): = (1-\k )F - \k C\vp + (\k +1)\Psi $, and $0 < \delta_1 < 1$ small. 
For the first term on the RHS of equation (\ref{w2p-010}), we use the integration by parts 
\begin{eqnarray}
\label{w2p-011}
&& - \int_X \theta_\ep^2 v^{p-1} e^{B} \frac{1}{1-\k} \Delta \big( (1-\k)F - \k C\vp + (1+\k)\Psi  \big)\omega_D^n
\nonumber\\
&= & - \int_X \theta^2_\ep \frac{ v^{p-1}  e^{B }   }{\k -1} |\nabla\left( (1-\k)F -\k C\vp +(1+\k)\Psi   \right)|^2 \omega_D^n
\nonumber\\
&-& \int_X\frac{p-1}{\k -1} \theta_\ep^2 v^{p-2} e^{B} \nabla v\cdot \nabla \{ (1-\k)F -\k C\vp +(1+\k)\Psi \} \omega_D^n
\nonumber\\
&-& 2\int_X \frac{  v^{p-1} e^{B} }{\k -1} \nabla\theta_\ep \cdot \{\theta_\ep \nabla \big((1-\k)F -\k C\vp +(1+\k)\Psi \big) \}\omega_D^n
\nonumber\\
&\leq& \frac{(p-1)^2}{2(1-\delta_1)(\k-1)} \int_X \theta_\ep^2 v^{p-3} e^{B} |\nabla v|^2\omega_D^n + \delta_1^{-1} \rm{II}_\ep
\nonumber\\
&\leq&  \frac{(p-1)^2}{2(1-\delta_1)(\k-1)} \int_X \theta_\ep^2 v^{p-2}  |\nabla_{\vp} v|_{\vp}^2\omega_{\vp}^n + \delta_1^{-1} \rm{II}_\ep,
\nonumber\\
\end{eqnarray}
where the last term is 
$${\rm{II}}_\ep: = \frac{1}{\k -1}\int_X v^{p-1} e^B |\nabla \theta_\ep|^2 \omega_D^n. $$
Here we used the inequality 
\begin{equation}
\label{w2p-0110}
- 2 \frac{  v^{p-1} e^{B} }{\k -1} \nabla\theta_\ep \cdot ( \theta_\ep \nabla B ) \leq \delta_1 \frac{\theta_\ep^2  e^B v^{p-1}}{\k-1} |\nabla B|^2
 + \frac{v^{p-1}e^B}{\delta_1(\k-1)} |\nabla \theta_\ep|^2,
\end{equation}
and 
\begin{equation}
\label{w2p-0115}
-\frac{p-1}{\k -1} v^{p-2} \nabla v\cdot \nabla B \leq \frac{(1-\delta_1) v^{p-1}}{2(\k -1)}|\nabla B|^2 + \frac{(p-1)^2 v^{p-3}|\nabla v|^2}{2(1-\delta_1)(\k -1)}.
\end{equation}
Picking up $\delta_1: = \frac{1}{2}$, we have 
\begin{eqnarray}
\label{w2p-012}
- \int_X \theta_\ep^2 v^{p-1} e^{A} \Delta F \omega^n_\vp &\leq& \frac{(p-1)^2}{(\k-1)} \int_X \theta_\ep^2 v^{p-2}  |\nabla_{\vp} v|_{\vp}^2\omega_{\vp}^n + 2{\rm{II}_\ep}
\nonumber\\
& + &   \int_X \theta_\ep^2 v^{p-1}e^{A}  \frac{  \k C \Delta\vp -(1+\k)\Delta\Psi}{\k -1}\omega_{\vp}^n.
\end{eqnarray}
Plugging equation (\ref{w2p-012}) back to (\ref{w2p-009}), we have 
\begin{eqnarray}
\label{w2p-014}
&& \int_X \left( p-1-\delta - \frac{(p-1)^2}{\k -1} \right) \theta_\ep^2 v^{p-2} |\nabla_{\vp}  v|_{\vp}^2\omega_{\vp}^n 
\nonumber\\
&\leq&  -\int_X \frac{\k C}{4} \theta_\ep^2 v^{p} (\tr_\vp g) \omega^n_\vp
+\int_X \theta^2_\ep v^{p-1} e^A \left( \k C_{18}(n+\Delta\vp) + \frac{\k C}{\k -1}\Delta\vp   \right)\omega_\vp^n
\nonumber\\
&+& \int_X \theta_\ep^2 v^{p-1} e^A \frac{\k +1}{\k -1} (- \Delta\Psi) \omega^n_{\vp} + \delta^{-1} {\rm I}_\ep + 2\rm{II}_\ep.
\nonumber\\
\end{eqnarray}
Picking up $\k \geq 2$, we see
$$\k C_{18}(n+\Delta\vp) + \frac{\k C}{\k -1}\Delta\vp \leq \k (C_{18}+C) (n + \Delta\vp),$$
and by equation (\ref{w2p-0015})
$$ \theta_{\ep}^2 \frac{\k +1}{\k -1} (- \Delta\Psi) \leq  3n C_{15} \theta_{\ep}^2\leq C_{19} \theta_\ep^2 (n+\Delta\vp),$$
where the constant $C_{19}$ depends on the uniform lower bound of $F$.
Eventually, we come up with
\begin{eqnarray}
\label{w2p-015}
&& \int_X \left( p-1-\delta - \frac{(p-1)^2}{\k -1} \right) \theta_\ep^2 |\nabla_{\vp}  v|_{\vp}^2\omega_{\vp}^n + \int_X \frac{\k C}{4} \theta_\ep^2 v^{p} (\tr_\vp g) \omega^n_\vp
\nonumber\\
&\leq& \int_X \k \theta_\ep^2 (C_{18} + C+ C_{19}) v^p \omega_\vp^n + \delta^{-1} {\rm I}_{\ep} + 2 \rm{II}_{\ep}.
\nonumber\\
\end{eqnarray}

Take the number $\k: = \max \{2, 10p/9 \}$, and then by our choice of $\delta$, we have 
$$ p-1 -\delta - \frac{(p-1)^2}{\k-1} \geq 0. $$
Drop the positive term in equation (\ref{w2p-015}) involving $|\nabla_\vp v|^2_{\vp}$, and then one obtains
\begin{equation}
\label{w2p-016}
 \int_X \frac{\k C}{4} \theta_\ep^2 v^{p} (\tr_\vp g) \omega^n_\vp \leq  C_{20} \int_X \k \theta_\ep^2  v^p \omega_\vp^n + \delta^{-1} {\rm I}_{\ep} + 2 \rm{II}_{\ep}.
\end{equation}
For fixing $\delta$, we let $\ep\rightarrow 0$, and then the two error terms converges to zero by equation (\ref{w2p-007}),
and we have 
\begin{equation}
\label{w2p-017}
 \int_Y   v^{p} (\tr_\vp g) \omega^n_\vp \leq  C_{20} \int_Y   v^p \omega_\vp^n
\end{equation}
Moreover, since $(n+\Delta\vp) \leq e^F (\tr_\vp g)^{n-1}$ on $Y$, we have from the definition of $v$ that 
\begin{eqnarray}
\label{w2p-018}
&& \int_Y e^{\left( \frac{n-2}{n-1} - \k p \right)F + p(\k+1)\Psi - p\k C\vp } (n+\Delta\vp)^{p + \frac{1}{n-1}} \omega_D^n
\nonumber\\
&\leq&  C_{20} \int_Y e^{\left( 1 - \k p \right)F + p(\k+1)\Psi - p\k C\vp } (n+\Delta\vp)^{p } \omega_D^n
\end{eqnarray}
Let $C_{21}$ be a constant such that $||\vp||_0, || F ||_0, ||\Psi||_0 < C_{21}$, 
and then we have 
\begin{equation}
\label{w2p-019}
 \int_Y (n+\Delta\vp)^{p + \frac{1}{n-1}} \omega_D^n \leq C_{22} e^{(\k p +p) C_{21}} \int_Y (n+\Delta\vp)^{p}\omega_D^n,
\end{equation}
where the uniform constant $C_{22}$ does not depend on $p$ or $\k $.

By induction, we can conclude our theorem if there exists a $p_0 >1$ such that the integral 
$$\int_Y (n+\Delta\vp)^{p_0}\omega_D^n$$
is uniformly bounded,
and we claim that this is true for $p_0 = 1+\frac{1}{n-1}$.

In fact, take a sequence of real numbers $1< p_i < 1.5 $ such that $p_i \searrow 1$. 
Then for each $p_i$, we can take $\k =2$, and then there exists a constant $C_{23}$ to satisfy 
$$  C_{22} e^{(2 p_i +2) C_{21}} \leq C_{23},$$ 
for all $i$. 

By the H\"older inequality, we have  
\begin{eqnarray}
\label{w2p-020}
\int_Y (n+\Delta\vp)^{1+ \frac{1}{n-1}} \omega^n_D &\leq& \int_Y (n+\Delta\vp)^{ p_i + \frac{1}{n-1}} \omega^n_D
\nonumber\\
&\leq& C_{23} \int_Y (n+\Delta\vp)^{ p_i} \omega^n_D,
\end{eqnarray}
but the last term is converging to the following by the dominate convergence theorem as $i\rightarrow +\infty$
\begin{equation}
\label{w2p-021}
\int_X (n + \Delta\vp) \omega^n_D =  \int_X ( \omega + dd^c( \psi_D + \varphi) ) \wedge ( \omega+ dd^c\psi_D)^{n-1} = n.
\end{equation}
Here we used Stoke's theorem for $L^{\infty}$ quasi-plurisubharmonic functions.
Therefore, for all $i$ large enough, we have 
$$\int_Y (n+\Delta\vp)^{ p_i} \omega^n_D \leq n+1. $$
The claim is proved, and our result follows.

\end{proof}

\section{The Laplacian estimate}
\label{sec-006}
Recall our notations in the previous section. 
Let $(\psi, G)$ be a $\cC^{2,\a,\b}$-conic cscK pair on $(X,D)$, with respect to the background metric $\omega_\b$,
and a new pair $(\vp, F)$ be its reformulation with respect to Donaldson's metric $\omega_D$, i.e. they satisfy equations (\ref{w2p-004}) and (\ref{w2p-005}).
Their relations are 
$\vp= \psi  - \psi_D + \psi_\b$,
$ F = G +h$.
However, it is important that our target metric remains the same as
$$\omega_{\vp}: = \omega_{D} + dd^c\vp = \omega_\b + dd^c\psi.$$ 
Moreover, we adapt to the following conventions: 
\begin{itemize}

\item denote $g, \nabla, \Delta$ as the Riemnannian metric, gradient, and Laplacian 
with respect to $\omega_D$;

\item denote $g_\b, \nabla_\b, \Delta_\b$ with respect to the background metric $\omega_\b$;

\item denote $g_\vp, \nabla_\vp, \Delta_\vp$ with respect to $\omega_\vp$.
\end{itemize}

The two background metrics $\omega_D$ and $\omega_\b$ are actually quasi-isometric to each other on $X$,
and then there exists a uniform constant $C_{24}$ to satisfy 
$$C_{24}^{-1} \tr_{\omega_D} \omega_\vp \leq \tr_{\omega_\b}\omega_\vp \leq C_{24} \tr_{\omega_D}\omega_\vp, $$
equivalently 
\begin{equation}
\label{c2-001}
C_{24}^{-1} (n + \Delta \vp) \leq  (n + \Delta_\b \psi) \leq C_{24} (n+ \Delta \vp).
\end{equation}

Before proceeding to the $C^2$-estimate, we need a different version of the Sobolev inequality for conic metrics. 
Recall the set $Y: = X\sm \text{Supp}(D)$, and then we proved the following.
\begin{lemma}
\label{c2-lem-sob}
Let $u$ be any smooth function on $Y$ satisfying $\sup_{Y}|u| < +\infty$. For any $1< p\leq 2$ and $q = \frac{2np}{2n-p}$, we have 
\begin{equation}
\label{c2-sob}
\left( \int_Y |u|^q \omega_D^n \right)^{\frac{1}{q}} \leq C_{sob,D} \left\{  \left( \int_Y |\nabla_g u|^p_g \omega_D^n \right)^{\frac{1}{p}}   
+ \left( \int_Y |u|^p \omega_D^n \right)^{\frac{1}{p}}   \right\},
\end{equation}
for some uniform constant $C_{sob, D}$.
\end{lemma}
\begin{proof}
It is enough to argue in an open neighbourhood $U$ of a point $p\in \text{Supp}(D)$,
and the general case follows in the standard way by using a partition of unity. 
Suppose $ (z_1,\cdots,z_n)$ is a holomorphic coordinate chart on $U$, such that $p$ is its origin 
and the defining function of the support of the divisor is $\{ z_1\cdots z_d =0\}$.
Recall that Donaldson's polar coordinate is a bijection $\Xi: B_{1}(0) \rightarrow U$ as 
$$(\z_1,\cdots, \z_d, z_{d+1}, \cdots, z_n) \rightarrow (|\z_1|^{\frac{1}{\b_1}-1}\z_1,\cdots,  |\z_d|^{\frac{1}{\b_d}-1}\z_d, z_{d+1},\cdots,  z_n).$$

This map $\Xi$ is a bijection, diffemorphism outside of the divisor, but it is no longer holomorphic. 
Moreover, the pull back of the conic K\"ahler metric $\Xi^*\omega_D$ is quasi-isometric to the Euclidean metric on Donaldson's polar coordinate. 
Therefore, it is enough to prove the following inequality 
\begin{equation}
\label{c2-sob-1}
\left( \int_{B_1\sm D} u^q  \right)^{\frac{1}{q}} \leq C \left( \int_{B_1\sm D} |\nabla u|^p \right)^{\frac{1}{p}} + C \left( \int_{B_1\sm D} |u|^p \right),
\end{equation}
for some uniform constant $C$. 

Let $\theta_\ep$ be our previous cut-off function introduced on Donaldson's polar coordinate,
and then apply the Sobolev inequality, with exponent $p$, to the smooth function $\theta_\ep u$ on $B_1$ to have 
\begin{equation}
\label{c2-sob-2}
\left( \int_{B_1 } | \theta_\ep u |^q  \right)^{\frac{1}{q}} \leq C \left( \int_{B_1} |\nabla (\theta_\ep u) |^p \right)^{\frac{1}{p}} + C \left( \int_{B_1} |\theta_\ep u|^p \right).
\end{equation}
The only issue is on the gradient term while taking convergence of the above equation.
Thanks to the Minkowski inequality, this gradient term is controlled by 
\begin{equation}
\label{c2-sob-3}
\begin{split}
& \left( \int_{B_1} | \theta_\ep \nabla u |^p \right)^{\frac{1}{p}} +  \left( \int_{B_1} | u \nabla \theta_\ep  |^p \right)^{\frac{1}{p}}  \\
& \leq  \left( \int_{B_1} | \theta_\ep \nabla u |^p \right)^{\frac{1}{p}}  + || u ||_{L^{\infty}} \int_{B_1} C d\theta_\ep\wedge d^c\theta_\ep \wedge \omega_{Euc}^n.
 \end{split}
\end{equation}
The last term on the RHS of equation (\ref{c2-sob-3}) converges to zero as $\ep\rightarrow 0$,
and equation (\ref{c2-sob-1}) is proved.

\end{proof}

Then we can prove the following $C^2$-estimate for conic cscK metrics.

\begin{theorem}
\label{main-thm-3}
There exists $p_n > 0$ only depending on the dimension $n$, such that 
\begin{equation}
\label{c2-002}
 \max_X |\nabla_{\vp} G|^2_{\vp} + \max_X (n + \Delta_{\b}\psi) \leq C_{25},
\end{equation}
where the constant $C_{25}$ depends on the following 
$ ||\vp||_0 $, 
$ || G ||_0$,
$ || h ||_{0}$, 
$|| \psi_D||_0$, 
$||\psi_\b||_0$,
$ || n + \Delta\vp ||_{L^{p_n}(\omega_D^n)}$,
$ (X, \omega_D, \omega_\b )$ and $(D,\phi)$.
\end{theorem}

Since the two background metrics are quasi-isometric, and the functions $G, \vp, \psi$ are all in $\cC^{2,\a,\b}$,
it is enough to prove that the following estimate holds: 
\begin{equation}
\label{c2-003}
 \max_Y  |\nabla_{\vp} G|^2_{\vp} + \max_Y (n + \Delta\vp) \leq C_{25}.
\end{equation}

The reason to switch the two background metrics back and forth is as follows:
on the one hand, the \emph{Ricci} curvature $Ric(\omega_{\b})$ is smooth and uniformly bounded outside the divisor,
but its bisectional curvature is not completely clear for normal crossing divisors;
on the other hand, the growth of the bisectional curvature $R_{i\bar i j\bar j}(\omega_D)$ is clear,
but its Ricci curvature $Ric(\omega_D)$ is no longer bounded under the conic metric $\omega_D$.
In fact, the norm $|\d h |^2_{\omega_{cone}}$ is not bounded if $\b > 2/3$.

Since $G\in\cC^{2,\a,\b}$, the first term $|\nabla_\vp G|^2_\vp$ is an $L^{\infty}$ function on $X$.
Then we invoke Chen-Cheng's $C^2$-estimates \cite{CC2}, \cite{CC3}, and compute on $Y$ to obtain 
\begin{equation}
\label{c2-004}
\begin{split}
e^{- \frac{G}{2}} \Delta_\vp (e^{\frac{G}{2}} |\nabla_\vp G|^2_{\vp}) & \geq  2  \nabla_\vp G \cdot_\vp \nabla_\vp (\Delta_\vp G) 
 + g^{\bar q p}_\vp g^{\bar\b\a}_\vp \Theta_{p\bar\b} G_\a G_{\bar q}  \\
 & + \frac{1}{2} |\nabla_\vp G|^2_{\vp} ( -\ul{R}_{\b} + \tr_\vp \Theta) + g^{\bar q p}_\vp g^{\bar\b \a}_\vp G_{\a\bar q} G_{p\bar\b}.
\end{split}
\end{equation}
Here we used the fact $Ric(\omega_\b) = \Theta$ on $Y$, and then we have 
\begin{equation}
\label{c2-005}
\tr_\vp \Theta - \ul{R}_{\b} \geq - C_{26} ( 1 + \tr_\vp \omega_{\b}) \geq - C_{26} (1+ e^{-G}(n + \Delta_\b\psi)^{n-1}),
\end{equation}
and also 
\begin{equation}
\label{c2-006}
g^{\bar q p}_\vp g^{\bar\b\a}_\vp \Theta_{p\bar\b} G_\a G_{\bar q} \geq - C_{27} |\nabla_\vp G|^2_{\vp}( n+ \Delta_\b \psi )^{n-1}.
\end{equation}
Then we came up with the following by equation (\ref{c2-001})
\begin{equation}
\label{c2-007}
\begin{split}
\Delta_\vp (e^{\frac{G}{2}} |\nabla_\vp G|^2_{\vp}) & \geq  2 e^{ \frac{G}{2}}  \nabla_\vp G \cdot_\vp \nabla_\vp (\Delta_\vp G)
+ \frac{1}{C_{28}} g^{\bar q p}_\vp g^{\bar\b \a}_\vp G_{\a\bar q} G_{p\bar\b} \\
&  - C_{28} (1+ (n + \Delta\vp)^{n-1}) |\nabla_\vp G|^2_{\vp}.
\end{split}
\end{equation}

A key observation is that the positive term in equation (\ref{c2-007})
is actually the $L^2$-norm 
under the target metric $\omega_\vp$ of the complex Hessian of $G$,
and then it can be re-written in $\omega_D$-normal coordinate as  
$$ g^{\bar q p}_\vp g^{\bar\b \a}_\vp G_{\a\bar q} G_{p\bar\b} = |\ddbar G|^2_{g_{\vp}} = \frac{|G_{i\bar j}|^2}{(1+ \vp_{i\bar i})(1+ \vp_{j\bar j})}. $$

\subsection{The $C^2$ estimate}
From now on, we stick to the background metric $\omega_D$.
Let $\Psi_{\g}$ be the auxiliary function used in Guenancia-P\u aun's trick,
and we recall the following inequality 
$$ \Delta_{\vp} \log (n+ \Delta\vp)  \geq  - C_{16} \text{tr}_{\vp} g + \frac{\Delta F}{ n + \Delta\vp} - \Delta_\vp \Psi.$$
Then we have by equations (\ref{w2p-003}) and (\ref{w2p-0015})
\begin{equation}
\label{c2-008}
\begin{split}
e^{-2\Psi} \Delta_\vp (e^{2\Psi} (n+ \Delta\vp)) & \geq -C_{16}(n+ \Delta\vp)\tr_\vp g +  \Delta G \\
& +  \Delta h + (n+ \Delta\vp) \Delta_\vp \Psi \\
& \geq -(C_{16} + C_{17})(n+ \Delta\vp)\tr_\vp g  + \Delta G \\
& \geq - C_{29} (n+ \Delta\vp)^n - \frac{1}{C_{28}} \frac{|G_{i \bar i}|^2}{(1+ \vp_{i\bar i})^2}
\end{split}
\end{equation}
Here we used the inequality 
$$ \tr_{\omega_D} (C_{17}\omega_D + dd^c\Psi) \leq \tr_{\omega_D}\omega_{\vp} \cdot  ( C_{17}\tr_\vp g + \Delta_\vp\Psi). $$
Put 
$$ u: = e^{\frac{G}{2}} |\nabla_\vp G|_\vp^2 + (n + \Delta\vp) + 1.$$
Combining with equations (\ref{c2-007}) and (\ref{c2-008}), we obtain 
\begin{equation}
\label{c2-009}
\Delta_\vp u \geq 2 e^{ \frac{G}{2}}  \nabla_\vp G \cdot_\vp \nabla_\vp (\Delta_\vp G) - C_{30} (n+\Delta\vp)^{n-1}u.
\end{equation}

\begin{proof}[Proof of Theorem (\ref{main-thm-3})]
We will do integration by parts for the first term on the RHS of equation (\ref{c2-009}) as follows.
Let $p>0$, and we use the previous cut off function to have 
\begin{equation}
\label{c2-010}
\begin{split}
 2p \int_X \theta^2_\ep u^{2p-1} |\nabla_\vp u|^2_\vp \omega_\vp^n &= \int_X \theta_\ep^2 u^{2p} (-\Delta_\vp u) \omega_\vp^n  \\
 &- 2\int_X ( u \nabla_\vp\theta_\ep ) \cdot_\vp ( \theta_\ep\nabla_\vp u ) u^{2p-1} \omega_\vp^n,
 \end{split}
\end{equation}
and then we have by the Cauchy-Schwarz inequality 
\begin{equation}
\label{c2-011}
\begin{split}
p \int_X \theta^2_\ep u^{2p-1} |\nabla_\vp u|^2_\vp \omega_\vp^n  &\leq  
 C_{30}\int_X \theta_\ep^2 (n+ \Delta\vp)^{n-1} u^{2p+1} \omega_\vp^n \\
& -2 \int_X \theta_\ep^2 e^{\frac{G}{2}} \nabla_\vp G \cdot_\vp \nabla_\vp (\Delta_\vp G) u^{2p}\omega_\vp^n + p^{-1} {\rm{IV_\ep}}
\end{split}
\end{equation}
where the error term is 
$$ {\rm{IV}}_\ep: = \int_X   u^{2p+1} d\theta_\ep\wedge d^c\theta_\ep\wedge\omega_{\vp}^{n-1}. $$
Then perform the integration by parts as 
\begin{equation}
\label{c2-012}
\begin{split}
 &  -2 \int_X \theta_\ep^2 e^{\frac{G}{2}} \nabla_\vp G \cdot_\vp \nabla_\vp (\Delta_\vp G) u^{2p}\omega_\vp^n = 
\int_X 4p\theta_\ep^2 u^{2p-1} e^{\frac{G}{2}} \Delta_\vp G  (\nabla_\vp G\cdot_\vp \nabla_\vp u ) \omega_\vp^n \\
& + \int_X 2\theta_\ep^2 u^{2p} e^{\frac{G}{2}} (\Delta_\vp G)^2 \omega_\vp^n 
+ \int_X \theta_\ep^2 u^{2p} e^{\frac{G}{2}} |\nabla_\vp G|^2_{\vp} \Delta_{\vp} G \omega_\vp^n + {\rm{V_\ep}},
\end{split}
\end{equation}
where the error term is 
$${\rm{V_\ep}}: = 4\int_X \theta_\ep e^{\frac{G}{2}} u^{2p} (\Delta_\vp G)   d \theta_\ep  \wedge d^c  G \wedge \omega_\vp^{n-1}.$$
This error $\rm{V}_\ep \rightarrow 0$ as $\ep\rightarrow 0$ by equation (\ref{w2p-0070}). 
Then we further use the Cauchy-Schwarz inequality to obtain 
\begin{equation}
\label{c2-014}
\begin{split}
&  4p \int_X \theta_\ep^2 u^{2p-1} e^{\frac{G}{2}} \Delta_\vp G  (\nabla_\vp G\cdot_\vp \nabla_\vp u ) \omega_\vp^n 
\leq \frac{p}{2} \int_X \theta_\ep^2 u^{2p-1} |\nabla_\vp u|^2_\vp \omega_\vp^n \\
&  + 8 p\int_X \theta_\ep^2  u^{2p} e^{\frac{G}{2}} (\Delta_\vp G)^2 \omega_\vp^n.
\end{split}
\end{equation}
Eventually we have 
\begin{equation}
\label{c2-015}
\begin{split}
& -2 \int_X \theta_\ep^2 e^{\frac{G}{2}} \nabla_\vp G \cdot_\vp \nabla_\vp (\Delta_\vp G) u^{2p}\omega_\vp^n 
\leq   \frac{p}{2} \int_X \theta_\ep^2 u^{2p-1} |\nabla_\vp u|^2_\vp \omega_\vp^n 
\\
& + (8p +2)  \int_X \theta_\ep^2  u^{2p} e^{\frac{G}{2}} (\Delta_\vp G)^2 \omega_\vp^n
+ \int_X \theta_\ep^2 u^{2p+1} \Delta_\vp G \omega_\vp^n + {\rm{V_\ep}}.
\end{split}
\end{equation}
Combined with equation (\ref{c2-011}), one obtains
\begin{equation}
\label{c2-016}
\begin{split}
&  \frac{p}{2} \int_X \theta^2_\ep u^{2p-1} |\nabla_\vp u|^2_\vp \omega_\vp^n  \leq  
 C_{30}\int_X \theta_\ep^2 (n+ \Delta\vp)^{n-1} u^{2p+1} \omega_\vp^n 
 \\
 & + (8p + 2) \int_X \theta_\ep^2 u^{2p} e^{\frac{G}{2}} (\Delta_{\vp} G)^2 \omega_{\vp}^n 
+  \int_X \theta_\ep^2 u^{2p+1} \Delta_\vp G \omega_\vp^n + {\rm{V_\ep}} + p^{-1} {\rm{IV_\ep}} 
\end{split}
\end{equation}
For fixed $p$ and any $\ep$ small enough, the RHS of equation (\ref{c2-016}) is bounded from the above by 
\begin{equation}
\label{c2-017}
\begin{split}
&   C_{30}\int_Y  (n+ \Delta\vp)^{n-1} u^{2p+1} \omega_\vp^n 
+ (8p + 2) \int_Y u^{2p} e^{\frac{G}{2}} (\Delta_{\vp} G)^2 \omega_{\vp}^n
\\
& +  \int_Y u^{2p+1} | \Delta_\vp G| \omega_\vp^n  +1
\end{split}
\end{equation}
Therefore, the LHS is uniformly bounded from the above, and it is monotony increasing to 
$$\frac{p}{2} \int_Y u^{2p-1} |\nabla_\vp u|^2_\vp \omega_\vp^n, $$ 
as $\ep\rightarrow 0$. Hence we can take $\ep\rightarrow 0$ simultaneously on the both sides of equation (\ref{c2-016}) to obtain 
\begin{equation}
\label{c2-018}
 p \int_Y u^{2p-1} |\nabla_\vp u|^2_\vp \omega_D^n \leq C_{31}(p+1) \int_Y  (n + \Delta\vp)^{2n-2} u^{2p+1} \omega^n_{D},
\end{equation}
and then we have
\begin{equation}
\label{c2-019}
\int_Y |\nabla_\vp (u^{p + \frac{1}{2}}) |^2_\vp \omega^n_D \leq \frac{(p+\frac{1}{2})^2(p+1)}{p}  \int_Y C_{31}(n + \Delta\vp)^{2n-2} u^{2p+1}\omega_D^n.
\end{equation}
Take $0<\delta<2$ and $v: = u^{p+\frac{1}{2}}$. For $p\geq \frac{1}{2}$, we apply H\"older's inequality as in Chen-Cheng \cite{CC1} to obtain
\begin{equation}
\label{c2-020}
\left( \int_{Y} | \nabla v|^{2-\delta} \omega_D^n \right)^{\frac{2}{2-\delta}} \leq  p^2 K_\delta C_{32}  \left( \int_Y  v^{2+\delta} \omega_D^n \right)^{\frac{2}{2+\delta}},
\end{equation}
where the constant is 
$$K_\delta: = n^{\frac{\delta}{2-\delta}} \left(   \int_Y (n+\Delta\vp)^{\frac{2-\delta}{\delta}} \omega_D^n    \right)^{\frac{\delta}{2-\delta}} \cdot 
\left( \int_Y (n + \Delta\vp)^{\frac{(2n-2)(2+\delta)}{\delta}} \omega_D^n \right)^{\frac{\delta}{2+\delta}}. $$

Here $v$ is an $L^{\infty}$ function on $X$,
and then we can invoke Lemma (\ref{c2-lem-sob}),
with the Sobolev exponent $p = 2-\delta$ to have 
$$ || v ||_{L^{\mu} (\omega_D^n)} \leq C_{sob, D} \left\{  ||\nabla v||_{L^{2-\delta}(\omega_D^n)} + ||v||_{L^{2-\delta}(\omega_D^n)}     \right\}, $$
with $\mu: = \frac{2n(2-\delta)}{2n-2+\delta}$.
Therefore, we eventually obtain the following by the H\"older inequality 
\begin{equation}
\label{c2-021}
\left( \int_Y u^{p+\frac{1}{2}} \omega_D^n \right)^{\frac{2}{\mu}} \leq  p C_{32,\delta} \left( \int_Y u^{(p+\frac{1}{2})(2+\delta)} \omega_D^n \right)^{\frac{2}{2+\delta}}.
\end{equation}
Pick $\delta$ small enough to satisfy 
$$ \frac{2n(2-\delta)}{2n - 2+\delta} > 2+\delta, $$
and then by the standard iteration technique, we have 
\begin{equation}
\label{c2-022}
|| u ||_{L^{\infty}} \leq C_{33,\delta} || u ||^{\frac{1}{2+\delta}}_{L^1(\omega_D^n)} || u ||^{\frac{1+\delta}{2+\delta}}_{L^{\infty}},
\end{equation}
where the constant $C_{33,\delta}$ is uniformly bounded if $C_{32,\delta}$ is.
Therefore, the $L^{\infty}$-norm of $u$ is controlled by the $L^{1}(\omega_D^n)$-norm of $u$.
It is easy to see that $(n+\Delta\vp)\in L^1(\omega_D^n)$,
and we claim that $e^{\frac{G}{2}} |\nabla_\vp G|^2_{\vp} \in L^1_{\omega_D^n}$.

In fact, the following integral is zero by introducing the cut off function $\theta_\ep$ and use equation (\ref{w2p-0070}) to let $\ep\rightarrow 0$
\begin{equation}
\label{c2-023}
\frac{1}{2}\int_Y \Delta_\vp (G^2) \omega_\vp^n = \int_Y  e^{G}|\nabla_\vp G|^2_{\vp}\omega_\b^n +  \int_Y G e^G (-\ul{R}_{\b} + \tr_{\vp}\Theta) \omega_{\b}^n,
\end{equation}
and then we have 
\begin{equation}
\label{c2-24}
\int_Y e^{\frac{G}{2}} |\nabla_\vp G|^2_\vp \omega_D^n \leq C_{34}\int_Y (1+ \tr_\vp g_\b)\omega_{\b}^n \leq C_{34}(n+1),
\end{equation}
since the two background metrics $\omega_D$ and $\omega_\b$ are quasi-isometric.
\end{proof}

\begin{rem}
\label{c2-rem-001}
According to our proof, all the a priori estimates, including the $C^0$-estimate, non-degeneracy estimate, $W^{2,p}$ and $C^2$ estimates, 
still hold if we only assume that the cscK pair $(\vp, F)\in\cC^{1,\bar 1}_\b$ in the beginning.  
\end{rem}

\section{The twisted case}
\label{sec-007}
Let $(X,D)$ be a log smooth klt pair, and $D: = \sum_{k=1}^d (1-\b_k) D_i$ as before. 
We consider a slightly different version of the conic cscK metric in this section. 

Fix a closed $(1,1)$ form $\tau_0$ on $Y$, such that $| \tau_0 |_{\omega_D}$ is an $L^{\infty}$ function on $X$. 
Let $(\vp, F, f)$ be a triple of function in the space $\cC^{2,\a,\b}(X,D)\bigcap C^{\infty}(Y)$, and we consider the following equation 
\begin{equation}
\label{tw-001}
(\omega_\b + dd^c\vp)^n = e^{F}\omega_\b^n;
\end{equation}
\begin{equation}
\label{tw-002}
\Delta_\vp F = \tr_\vp(\Theta - \tau) - R,
\end{equation}
and we furhter assume the following conditions: 
\begin{itemize}
\item $\tau: =\tau_0 + dd^c f \geq 0$;
\item $\sup_X f =0$ and $ e^{-f} $ is uniformly bounded in $L^{p_0}(\omega_\b^n)$-norm, for some $p_0 >1$;
\item $R$ is a uniformly bounded function.  
\end{itemize}

For later use, we also re-write equation (\ref{tw-002}) as 
\begin{equation}
\label{tw-003}
\Delta_\vp (F+f) = \tr_\vp(\Theta - \tau_0) - R.
\end{equation}

Now we are going to prove all the a priori estimates for this twisted equation.
However, the difficulty is that we do not have a uniform upper bound for the $(1,1)$ form $\tau$.
Therefore, it is not reasonable to require the Laplacian estimate anymore. 

The proofs of the following a priori estimates are very similar with our previous arguments.
Therefore, we will only sketch the proof and emphasis the place where the twisting function $f$ brings a change.

\subsection{The $C^0$-estimate}

In this section, we do need to assume the positivity of $\tau$ as in Chen-Cheng \cite{CC3}. 
Let $\psi_1$ be the $\cC^{1,\bar 1}_\b$-conic auxiliary function constructed  in equation (\ref{pe-002}),
and we have the following.
\begin{lemma}
\label{tw-lem-001}
For any $\ep_0 >0$ small enough, there exists a constant $C_{35}$ to satisfy 
$$ F + f + \ep_0\psi_1 - 4(\max_{X} |\Theta - \tau_0|_{g} + 1  )\vp \leq C_{35}, $$
where the constant depends on
$$ C_{35}: = C_{35}\left(\ep_0, \int_X F e^F \omega_{\b}^n,\ \max_X |\tau_0|_g, || R ||_0, \omega_\b, X,D,\phi,\b \right).$$
\end{lemma}
\begin{proof}
Let $\psi_2: = \sum_{k=1}^d |s_k|^{\g_k}$ be the potential of Donaldson's metric $\omega_{D,\g}$ with angles $\g_k < \b_k$ along each $D_k$.
Denote a function $A(\vp, F, f)$ by 
$$ A(\vp, F, f): = F + f + \ep_0\psi_1 + \ep\psi_2 - \lambda\vp, $$
where $\lambda: = 4(\max_X |\Theta - \tau_0|_{g_\b} + 1)$, $\ep\in (0,1)$ and $\ep_0 >0$ is an arbitrary small number.
From equation (\ref{tw-003}), we compute on $Y$ as  
\begin{equation}
\label{tw-004}
\begin{split}
\Delta_\vp A &= \tr_\vp (\Theta - \tau_0) - R + \ep_0\Delta_\vp \psi_1 + \ep \Delta_\vp \psi_2 - \lambda n + \lambda\tr_\vp g_\b
\\
& \geq (\lambda - \ep_0 - \ep N_1 -  \lambda /4) \tr_\vp g_{\b} +  \ep_0 I_{\Phi}^{-\frac{1}{n}} \Phi^{\frac{1}{n}}(F) - (R + \lambda n)
\\
& \geq \frac{\lambda}{2} \tr_\vp g_{\b} + \ep_0 I_{\Phi}^{- \frac{1}{n}} \Phi^{\frac{1}{n}}(F) - C_{36},
\end{split}
\end{equation}
where the constant $C_{36}$ only depends on $|| R ||_0$, $|\Theta|_g$ and $|\tau_0|_g$.

Let the point $p$ be the maximum point of the function $u: = e^{\delta A}$, with $\delta: = (2n\lambda)^{-1}\a$.
Suppose $\eta_{p}$ is a cut off function in a coordinate ball $B_d(p)$ with radius $d$ centred at $p$,
such that $\eta_p (p) =1$ and $\eta_p = 1-\theta$ outside $B_{d/2}(p)$ for some $0 < \theta <1$.
Taking $\theta $ so small that it satisfies 
$$ \frac{(1-\theta)\a}{4n} - \frac{4\theta}{d^2} - \frac{4\theta^2}{d^2(1-\theta)} \geq 0, $$
we conclude with the following inequality on $B_d(p) \bigcap Y$
\begin{equation}
\label{tw-005}
\Delta_\vp(u\eta_p)\geq e^{\delta A}  \delta \eta_p \left( \ep_0 I_{\Phi}^{- \frac{1}{n}} \Phi^{\frac{1}{n}}(F) - C_{36}  \right).
\end{equation}

Moreover, since $f\in\cC^{2,\a,\b}$, the function $u\eta_p$ is strictly subharmonic in 
an open neighbourhood of the divisor, as we proved in Lemma (\ref{c0-lem-002}).
Then the upper contact set $\tg^+_{(u\eta_p)}$ is contained in an open subset $V_1$ of $B_d(p)$,
such that $V_1$ is disjoint from the divisor with a positive distance.  
Therefore, we apply GAMP to the function $u\eta_p$ on $B_d(p)$ to have 
\begin{equation}
\label{tw-006}
\begin{split}
& e^{\delta A}\eta_p (p) \leq \sup_{\d B_d(p)} e^{\delta A} \eta_p  
\\
& + C_n d \left( \int_{B_d(p)} e^{2F} e^{2n\delta A} \left\{ (\ep_0 I_{\Phi}^{- \frac{1}{n}} \Phi^{\frac{1}{n}}(F) - C_{36} )^-  \right\}^{2n} \omega^n  \right)^{\frac{1}{2n}}.
\end{split} 
\end{equation}
Since $f\leq 0$, $\psi_1\leq 0$, and $| \psi_2 | \leq 1$, 
the integral on the RHS of equation (\ref{tw-006}) is controlled by 
\begin{equation}
\label{tw-007}
\begin{split}
& C^{2n}_{36}\int_{B_d(p) \bigcap \{ F \leq C_{37} \}} e^{2F + 2n\delta F} e^{-2n\delta\lambda\vp} \omega^n 
\\
& \leq C^{2n}_{36} e^{(2+ 2n\delta)C_{37}} \int_{B_d(p)} e^{-\a\vp}\omega^n \leq C_{38},
\end{split}
\end{equation}
for some uniform constant $C_{37}$ depending on $\ep_0$, $\Phi$, $I_{\Phi}$ and $C_{36}$.  
Since $\eta_p = 1-\theta$ on $\d B_d(p)$, our result follows. 

\end{proof}

\begin{corollary}
\label{tw-cor-001}
There exists a constant $C_{37}$ such that 
$$  F+ f \leq C_{37}; \ \ \ ||\vp ||_0\leq C_{37}; \ \ \ || \psi_1 ||_0\leq C_{37}. $$
Here  the constant $C_{37}$ depends 
on the same things as the constant $C_{35}$ in Lemma (\ref{tw-lem-001}),
and also on $p_0$ and $\int_X e^{-p_0 f}\omega_\b^n$.
Moreover, it is uniformly bounded if $p_0 \geq 1 + \ep$ for some small $\ep>0$.
\end{corollary} 
\begin{proof}
We obtain from Lemma (\ref{tw-lem-001}) that 
$$  \a\ep_0^{-1}( F+ f ) \leq -\a\psi_1 + \ep_0^{-1} \a C_{35}.$$
For any $p$, we chose $\ep_0$ small enough to have $p: = \ep_0^{-1}\a$,
and then Proposition (\ref{app-prop-001}) implies that there exists a uniform constant $C_{38}$(may depend on $\b$) to satisfy,
\begin{equation}
\label{tw-008}
\int_X e^{p (F+f)} \omega_{\b}^n \leq C_{38},
\end{equation}
for every conic angle $\b$.

Let $\psi: = \sum_{k=1}^d(1-\b_k) \log |s_k|^2 $,
and then we have from H\"older's inequality 
\begin{equation}
\label{tw-009}
\begin{split}
&  || e^{F -\psi} ||_{L^{1+\ep}}^{1+\ep} =
\int_X  \frac{e^{(1+\ep)F - \ep\psi}dV}{\prod_{k=1}^d |s_k|^{2(1-\b_k)}} = \int_X  \frac{e^{(1+\ep)(F+f) - \ep\psi - (1+\ep) f}dV}{\prod_{k=1}^d |s_k|^{2(1-\b_k)}}
\\
& \leq 
\left(  \int_X e^{-p_0 f} \omega_\b^n \right)^{\frac{1+\ep}{p_0}} 
\left( \int_X e^{\frac{p_0 (1+\ep)}{p_0 -1 -\ep}(F+f)} e^{-\frac{p_0\ep}{p_0 -1 -\ep}\psi} \omega_\b^n \right)^{1-\frac{1+\ep}{p_0}}.
\end{split}
\end{equation}
Pick up $\ep =  \frac{p_0 -1}{m}$ for some larger integer $m$, 
and then the second integral on the RHS of the above inequality is equal to 
$$ \int_X \frac{ e^{q (F+f)} dV}{ \left( \prod_{k=1}^d |s_k|^2 \right)^{(1-\b_k) (1+ \frac{p_0}{m-1}) }}, $$
for $q: = \frac{p_0(1+\ep)}{p_0 -1 -\ep} $.
Therefore, it is uniformly bounded by equation (\ref{tw-008}) with a slightly smaller angle $\b'$, where $\b'_k = (1+\frac{p_0}{m-1})\b_k - \frac{p_0}{m-1} $.
Then we conclude the potential estimates as 
$$  ||\phi ||_0, \ || \psi_1 ||_0 \leq C_{37},$$
and the upper bound of $F+f$ follows from Lemma (\ref{tw-lem-001}) again. 
\end{proof}

The next step is to prove the lower bound of $F+f$. 
\begin{lemma}
\label{tw-lem-002}
There exists a uniform constant $C_{39}$ such that 
$$ F+ f \geq - C_{39}, $$
and this constant has the same dependence as $C_{37}$, with also $||\vp||_0$.
\end{lemma}
\begin{proof}
Take a function $A_3(F, f, \vp): = - F -f -\lambda\vp + \ep\psi_2$,
and put $u: = e^{\delta A_3}$. Pick up the constants as 
$$  \lambda: = 4(\max_X|\Theta - \tau_0 |_g + 1);\ \ \ \delta = \frac{p_0}{2n(p_0-1)};\ \ \  \ep = \frac{1}{N_1}.$$
Then we have
\begin{equation}
\label{tw-010}
- \Delta_\vp A_3 \leq  -\tr_\vp \omega_\b + || R ||_0 + \lambda n.
\end{equation}

Assume that $u$ achieves its maximum at the point $x$ on the manifold.
We can consider the function $u\eta_x$,
with a suitable chosen cut-off function $\eta_{x}$ centred at $x$.
Following the same calculation as in Lemma (\ref{nd-lem-001}), we have 
$$ \Delta_\vp (u\eta_x) \geq - \delta e^{\delta A_3} \left(  || R ||_0 + \lambda n \right).  $$

Since $F, f, \vp\in \cC^{2,\a,\b}$, 
the function $u\eta_{x}$ is again strictly subharmonic near the divisor
Then we can apply GAMP locally on a coordinate ball $B_d(x)$ to have 
\begin{equation}
\label{tw-011}
\begin{split}
& e^{\delta A_3}\eta_{x}(x) \leq \sup_{\d B_d(x)} e^{\delta A_3}\eta_{x}
\\
&+  C_n d \left( \int_X e^{2F} e^{-2n\delta A_3} ( || R||_0 + \lambda n)^{2n}\omega^n \right)^{\frac{1}{2n}}.
\end{split}
\end{equation}
Therefore, the lower bound of $F+f$ is controlled by 
\begin{equation}
\label{tw-012}
\begin{split}
& C_{40}\int_X e^{(2-2n\delta )F - 2n\delta f } \omega^n
\\
& \leq C_{40}  \left(  \int_X e^F \omega^n \right)^{\frac{p_0 -2}{p_0 -1}} \left( \int_X e^{-p_0 f} \omega^n \right)^{\frac{1}{p_0-1}},
\end{split}
\end{equation}
and the first integral in the above equation is uniformly bounded as we can see from equation (\ref{tw-009}).

\end{proof}

\subsection{The $W^{2,p}$ estimate}
In order to consider this higher order estimate, we switch the background metric to Donaldson's metric as before. 

Let $(\psi, G, f)\in \cC^{2,\a,\b}$ be the tripe solution of the twisted equations with respect to the background metric $\omega_\b$,
i.e. they satisfy equations (\ref{tw-001}) and (\ref{tw-002}).
Then we define a new triple $(\vp, F, f)$ such that 
$$ \omega_\vp = \omega_D + dd^c\vp = \omega_\b + dd^c\psi, $$
and $F: = G +h$, for $h = \log\frac{\omega_\b^n}{\omega_D^n}$.
Hence they satisfy  the following equations:
\begin{equation}
\label{tw-014}
(\omega_D+ dd^c\vp)^n = e^F \omega_D^n;
\end{equation}
\begin{equation}
\label{tw-015}
\Delta_\vp F = \tr_\vp (\Theta - \tau + dd^c h ) - R,
\end{equation}
where $\tau = \tau_0 + dd^c f$.
Here $\vp, f$ are still in $\cC^{2,\a,\b}$, and $F$ is in $L^{\infty} (X)$, but  may no longer be in the space $\cC^{2,\a,\b}$. 

\begin{theorem} 
\label{tw-thm-001}
Assume $\tau \geq 0$. For any $p\geq 1$, there exists a constant $C_p$ satisfying 
$$ \int_X  e^{(p-1)f}( n + \Delta_\b \psi  )^p\omega_{\b}^n \leq C_p,$$
and this constant has the same dependence as $C_{37}$, also with $|| F+ f ||_0$, $||\vp ||_0$, and $p$.
\end{theorem}

Since $\omega_D$ and $\omega_{\b}$ are quasi-isometric on $X$, and their potentials $\psi_D, \psi_\b$
are uniformly bounded, it is enough to prove the following inequality
\begin{equation}
\label{tw-016}
 \int_Y  e^{(p-1)f}( n + \Delta  \vp  )^p\omega_{D}^n \leq C'_p,
\end{equation}
for some uniform constant $C'_p$.

\begin{proof}[Proof of Theorem (\ref{tw-thm-001})]
Let $\k >0, C>0, \de>0$ be constants to be determined later, and $\Psi_\g$ be the conic weight function as before. 
Define the following function:
$$  A(\vp, F, f): = -\k(F + \de f + C\vp) +(\k +1)\Psi, $$
and choose  
$$ C: = 8 (\max_X |\Theta -\tau_0|_g + \max_X |\tau_0|_g + C_{17} + C_{16} + 1). $$
Then we compute on $Y$ to have 
\begin{equation}
\label{tw-017}
\begin{split}
& \ \ \ \ e^{-A} \Delta_\vp \left(e^A (n + \Delta\vp) \right) 
\\
& \geq \frac{\k C}{4} \tr_\vp g (n+ \Delta\vp) + \Delta F + \k(1-\delta)\Delta_\vp f (n+\Delta\vp) - \k C_{40} (n+\Delta\vp).
\end{split}
\end{equation}

Let $p > 1$ and $\de_1: = \frac{p-1}{10}$. Denote $v: = e^A (n+\Delta\vp)$, and introduce the previous cut off function $\theta_\ep$ supported outside the divisor. 
Put 
$$B(\vp, F, f): = (1- \k) F - \k C\vp - \k\de f + (\k +1)\Psi,$$
and then we can play the same tricks on the integration by parts 
as in equations (\ref{w2p-008}) - (\ref{w2p-012}). 
Finally we obtain 
\begin{equation}
\label{tw-018}
\begin{split}
& \int_X \left( p-1 - \de_1 - \frac{(p-1)^2}{\k-1} \right)\theta_\ep^2 v^{p-2}|\nabla_\vp v|^2_{\vp} \omega_\vp^n
\\
& \leq -\int_X \frac{\k C}{4} \theta_\ep^2 v^p (\tr_\vp g) \omega_\vp^n + \int_X \theta_\ep^2 v^{p-1} e^A \left( \k C_{40}(n+\Delta\vp) + \frac{\k C}{\k -1} \Delta\vp  \right)\omega_\vp^n
\\
& + \int_X \theta_\ep^2 v^{p-1} e^A \frac{ \k+1}{\k -1} (-\Delta\Psi) \omega_{\vp}^n  + \de_1^{-1} {\rm{I}}_\ep + 2 {\rm{II}}_\ep 
\\
& + \int_X \theta_\ep^2 v^{p-1} e^A \left\{ -\k(1-\de) \Delta_\vp f(n+ \Delta\vp)  + \frac{\k\de}{\k -1} \Delta f      \right\}\omega_\vp^n,
\end{split}
\end{equation}
where the error terms are 
$$ {\rm{I}}_\ep: = \int_X v^p d\theta_\ep \wedge d^c\theta_\ep \wedge \omega_{\vp}^{n-1}, $$
and 
$$ {\rm{II}}_\ep: = \frac{1}{\k -1}\int_X v^{p-1} e^B  d\theta_\ep \wedge d^c\theta_\ep \wedge \omega_{D}^{n-1}.$$
Moreover, we have ${\rm{I}}_\ep, {\rm{II}}_\ep \rightarrow 0$ as $\ep \rightarrow 0$ from the property of $\theta_\ep$.
Take $\de: = \frac{\k -1}{\k}$, and then we see 
\begin{equation}
\label{tw-019}
\begin{split}
 -\k(1-\de) \Delta_\vp f(n+ \Delta\vp)  + \frac{\k\de}{\k -1} \Delta f  \leq \max_X |\tau_0|_g \tr_\vp g (n+\Delta \vp).
\end{split}
\end{equation}

Since $f\leq 0$, the lower bound of $F$ is controlled by $ - || F +f ||_0$, and 
then $(n + \Delta\vp)$ also has a uniform lower bound. 
Therefore, taking $\k: = \max\{2, \frac{10}{9}p \}$, the term on the LHS of equation (\ref{tw-018}) becomes positive.
Then we obtain 
\begin{equation}
\label{tw-020}
\int_X \frac{\k C}{8} \theta_\ep^2 v^p (\tr_\vp g)\omega_\vp^n \leq C_{41} \int_X \k \theta_\ep^2 v^p \omega_\vp^n + \de_1^{-1} {\rm{I}_\ep} + 2{\rm{II}}_\ep,
\end{equation}
Since $\k < \k C/8$ by our choice, we have the following by letting $\ep\rightarrow 0$
\begin{equation}
\label{tw-021}
\begin{split}
\begin{split}
& \int_Y  e^{\left( \frac{n-2}{n-1} - \k p     \right) F - p(\k -1)f  + p(\k+1)\Psi - p\k C\vp } (n+\Delta\vp)^{p + \frac{1}{n-1}} \omega_D^n
\\
& \leq C_{41} \int_Y e^{\left(  1- \k p    \right) F - p(\k -1)f  + p(\k+1)\Psi - p\k C\vp } (n+\Delta\vp)^{p } \omega_D^n.
 \end{split}
\end{split}
\end{equation}
Let $C_{42}$ be a bound of $|| \vp ||_0, || F + f ||_0, || \Psi ||_0$, and we further obtain 
\begin{equation}
\label{tw-022}
\begin{split}
& \int_Y e^{ \left( p - \frac{n-1}{n-2} \right) f } (n + \Delta \vp)^{ p + \frac{1}{n-1}} \omega_D^n 
\\
& \leq C_{43}e^{(p + \k p) C_{42}} \int_Y e^{(p-1)f} (n + \Delta \vp)^p \omega_D^n,
\end{split}
\end{equation}
for some uniform constant $C_{43}$ not depending on $p$ or $\k$.

Take $p: = 1 + \frac{k}{n-1}$, and we want to use induction on $k \geq 1$.
Then it is enough to prove that the integral 
\begin{equation}
\label{tw-023}
 \int_Y e^{  \frac{1}{n-1} f } (n + \Delta \vp)^{ 1 + \frac{1}{n-1}} \omega_D^n  
\end{equation}
is uniformly bounded. 
Let $ 1< p_i < 1.5 $ be sequence of real numbers decreasing to $1$,
and then we have 
\begin{equation}
\label{tw-024}
\begin{split}
& \int_Y e^{ \left( p_i - \frac{n-1}{n-2} \right) f } (n + \Delta \vp)^{ p_i + \frac{1}{n-1}} \omega_D^n 
\\
& \leq C_{44} \int_Y e^{(p_i-1)f} (n + \Delta \vp)^{p_i} \omega_D^n,
\end{split}
\end{equation}
for some uniform constant $C_{44}$ not depending on $p_i$ or $\k$.
Since $f\leq 0$, the LHS of equation (\ref{tw-024}) converges to to equation (\ref{tw-023})
by dominant convergence theorem, 
and the RHS of equation (\ref{tw-024}) converges to 
\begin{equation}
\label{tw-025}
C_{44} \int_Y  (n + \Delta \vp) \omega_D^n   = n C_{44},
\end{equation}
and our result follows.

\end{proof}

\begin{corollary}
\label{tw-cor-001}
For any $1< q < p_0$, there exists a constant $\tilde{C}_q$ satisfying 
$$ \int_X  ( \tr_{\omega_\b} \omega_\vp )^q \omega_\b^n \leq \tilde{C}_q.$$
Here the constant $\tilde{C}_q$ has the same dependence as $C_{37}$, also with $|| F + f||_0$, $|| \vp ||_0$ and $q$. 
Moreover, it is uniformly bounded in $q$ if $q$ is bounded away from $p_0$. 
\end{corollary}
\begin{proof}
This follows from Theorem (\ref{tw-thm-001}) and H\"older's inequality.
Pick up $s: = \frac{p_0(q-1)}{p_0 -1}$, and then we have as in Chen-Cheng \cite{CC3}
\begin{equation}
\label{tw-026} 
\begin{split}
& \int_X ( \tr_{\omega_\b} \omega_\vp )^q \omega_\b^n 
\\
& \leq \left(  \int_X e^{-p_0 f} \omega_\b^n     \right)^{\frac{s}{p_0}} 
\left(  \int_X e^{\frac{sp_0}{p_0 -s} f} (\tr_{\omega_\b} \omega_\vp)^{\frac{p_0q}{p_0 -s}} \omega_\b^n   \right)^{1-\frac{s}{p_0}}.
\end{split}
\end{equation}

\end{proof}

\subsection{The gradient $F$-estimate}
As we explained before, the $C^2$ estimate is not expected in the twisted case anymore.
Therefore, we do not need to switch our background metrics in the following proof of the partial $C^3$ estimate,
i.e. the gradient estimate of $F+f$.

Let $(\vp, F ,f )\in \cC^{2,\a,\b}$ be the tripe for the twisted equations, i.e. they satisfy equations (\ref{tw-001}) and (\ref{tw-002}). 
Then we have the following estimate for $W: = F + f$.

\begin{theorem}
\label{tw-thm-002}
There exists a constant $k_n$, depending only on $n$, such that $\forall p_0 > k_n$, we have 
$$ |\nabla_\vp W |^2_\vp \leq C_{45}.$$
Here the constant $C_{45}$ has the same dependence as $C_{37}$, also with $|| F+f ||_0$ and  $||\vp||_0$. 
\end{theorem}
\begin{proof}
Since $W= F +f \in \cC^{2,\a,\b}$, we can compute on $Y$ as in Chen-Cheng \cite{CC3} to have 
\begin{equation}
\label{tw-027}
\begin{split}
& e^{-\frac{W}{2}} \Delta_\vp ( e^{\frac{W}{2}} |\nabla_\vp W|^2_\vp ) \geq \frac{1}{2} | \nabla_\vp W |_\vp^2 \left(\tr_\vp (\Theta -\tau_0) - R \right) 
\\
& + 2 \nabla_\vp W \cdot_\vp \nabla_\vp (\Delta_\vp W) + | \ddbar W |^2_\vp + g^{\bar j i}_\vp g^{\bar\la \mu}_\vp \Theta_{i\bar\la} W_\mu W_{\bar j}
\\
& - \Re \left\{   g^{\bar j i}_\vp g^{\bar\la \mu}_\vp ( \tau_{0} )_{i\bar\la} W_\mu W_{\bar j}  \right\}.
\end{split}
\end{equation}
Moreover, we estimate 
\begin{equation}
\label{tw-028}
\begin{split}
\tr_\vp (\Theta -\tau_0) - R & \geq -C_{46} ( e^{-F} (n + \Delta\vp)^{n-1} + 1)
\\
& \geq - C_{47} ( (n + \Delta\vp)^{n-1} + 1).
\end{split}
\end{equation}
Here we used the uniform lower bound of $F$ in terms of $|| F +f ||_0$. 
In a similar way, we can estimate other terms in equation (\ref{tw-027}),
and eventually have 
\begin{equation}
\label{tw-029}
\begin{split}
& \Delta_\vp(e^{\frac{W}{2}} |\nabla_\vp W|^2_\vp) \geq  2 e^{\frac{W}{2}}\nabla_\vp W \cdot_\vp \nabla_\vp (\Delta_\vp W)
\\
& - C_{48} e^{\frac{W}{2}} |\nabla_\vp W|^2_\vp \left\{ (n + \Delta\vp)^{n-1} +1    \right\}.
\end{split}
\end{equation}

Put $u: = e^{\frac{W}{2}} |\nabla_\vp W|^2_\vp + 1$, and $\tilde U: = C_{48}\left\{ (n+ \Delta\vp)^{n-1} +1 \right\}$. 
We further have 
\begin{equation}
\label{tw-030}
\Delta_\vp u \geq 2 e^{\frac{W}{2}}\nabla_\vp W \cdot_\vp \nabla_\vp (\Delta_\vp W) - u \tilde U.
\end{equation}
Pick up $p >0$, and introduce the previous cut off function $\theta_\ep$, we play the same integration by parts trick as 
in equations (\ref{c2-011})-(\ref{c2-015}), and eventually obtain the following 
\begin{equation}
\label{tw-031}
\begin{split}
& \frac{p}{2} \int_X \theta_\ep^2 u^{2p-1} |\nabla_\vp u|^2_\vp \omega_\vp^n \leq  \int_X \theta_\ep^2 u^{2p+1} (\tilde U + (\Delta_\vp W)^2 +1)
\\
& + (8p +2) \int_X \theta_\ep^2 u^{2p+1} e^{\frac{W}{2}} (\Delta_\vp W)^2 \omega_\vp^n + {\rm{V}}_\ep + p^{-1} {\rm{IV}}_\ep,
\end{split}
\end{equation}
where the error terms are 
$$ {\rm{IV}}_\ep: = \int_X u^{2p+1} d\theta_\ep \wedge d^c \theta_\ep \wedge \omega_\vp^{n-1}; $$
and 
$${\rm{V}}_\ep:  = 4\int_X \theta_\ep e^{\frac{W}{2}} u^{2p} (\Delta_\vp W) d\theta_\ep\wedge d^c W \wedge \omega_\vp^{n-1}. $$
Since $W\in \cC^{2,\a,\b}$, the two errors ${\rm{IV}}_\ep, {\rm{V}}_\ep\rightarrow 0$ as $\ep\rightarrow 0$.
Then we get by taking the limit of $\ep$
\begin{equation}
\label{tw-032}
\begin{split}
&  \frac{p}{(p+ \frac{1}{2})^2 C_{49}} \int_Y |\nabla_\vp ( u^{p + \frac{1}{2}} )|^2_\vp \omega_\b^n 
   \leq    p \int_Y u^{2p-1} |\nabla_\vp u|^2_\vp  \omega_\vp^n
 \\
 & \leq 16(p +1)\int_Y  u^{2p+1} U e^{F} \omega_\b^n,
 \end{split}
\end{equation}
where $U: = \tilde U + (\Delta_\vp W)^2 + (\Delta_\vp W)^2 e^{\frac{W}{2}} + 1 $.
Let $\de > 0$ be a small number, and $v: = u^{p+\frac{1}{2}}$. 
We use H\"older inequality again to have 
\begin{equation}
\label{tw-033}
\left( \int_Y |\nabla v|^{2-\de} \omega_\b^n \right)^{\frac{2}{2-\de}} \leq \frac{ C_{50} K_\de L_\de (p+1)^3}{p}  \left( \int_Y v^{\frac{4}{2-\de}} \omega_\b^n  \right)^{\frac{2-\de}{2}},
\end{equation}
where the coefficients are 
$$ K_\de: = \left(  \int_Y (n+\Delta\vp)^{\frac{2}{\delta} -1} \omega_\b^n \right)^{\frac{\de}{2-\de}},$$
and 
$$ L_\de: = \left(    \int_Y U^{\frac{2}{\de}} e^{\frac{2F}{\de}} \omega_\b^n    \right)^{\frac{\de}{2}}. $$

Apply the conic version of Sobolev's inequality (Lemma (\ref{c2-lem-sob})) with exponent $2-\de$ to have 
\begin{equation}
\label{tw-034}
|| u^{p+\frac{1}{2}} ||^2_{L^{\mu}(\omega_\b^n)} \leq C_{51} \frac{(p+1)^3}{p}  (K_\de L_\de +1) || u^{p+\frac{1}{2}}||_{L^{\frac{4}{2-\de}}(\omega_\b^n)},
\end{equation}
where we assumed 
$$\mu: = \frac{2n(2-\de)}{ 2n -2 +\de} > \frac{4}{2-\de}. $$

This can be realised by choosing $\de = \frac{1}{2n}$ for $n >1$.
Eventually, we can use the Moser Iteration to control the $L^{\infty}$-norm of $u$,
provided that $L_\de, K_\de$ is uniformly bounded. 
This is also true from Theorem (\ref{tw-thm-001})
by choosing $p_0 \geq 8n(2n-1) +1$.

For the uniform control on the $L^1$-norm of $u$, we see 
\begin{equation}
\label{tw-035}
\int_X \theta_\ep( \Delta_\vp e^{\frac{W}{2}} ) \omega_\vp^n = \frac{1}{2}\int_X e^{\frac{W}{2}} d\theta_\ep \wedge d^c W \rightarrow 0,
\end{equation}
as $\ep\rightarrow 0$.
Therefore, we have 
\begin{equation}
\label{tw-036}
\begin{split}
& \int_Y e^{\frac{W}{2}} |\nabla_\vp W|^2_\vp \omega_\b^n \leq C_{52} \int_Y e^{\frac{W}{2}} |\nabla_\vp W|^2_\vp \omega_\vp^n
\\
& \leq C_{52} \int_Y 2 e^{\frac{W}{2}} (-\Delta_\vp W) \omega_\vp^n 
\\
& \leq C_{52} \int_Y( 1 + \tr_{\omega_\vp}\omega_\b) \omega_\vp^n  = C_{52}(n+1)
\end{split}
\end{equation}

\end{proof}

\begin{rem}
\label{tw-rem-001}
As before, all our estimates for the twisted equations only require $(\vp, F, f)\in \cC^{1,\bar 1}_\b$. 
\end{rem}

\section{Appendix}
In this section, we will consider the $\a$-invariant for plurisubharmonic(\emph{psh}) functions integrated against conic volume form
$$d\mu: = \omega_D^n = \frac{dV}{\prod_{k=1}^d |s_k|^{2-2\b_k}}.$$
Let $PSH(X,\omega)$ denote the space of all $\omega$-$psh$ functions on $X$, and the first observation is that 
they are all $L^1$ functions with respect to the measure $\mu$.

\begin{lemma}
\label{app-lem-001}
For any $\vp\in PSH(X,\omega)$, we have 
$$\int_X \vp d\mu > -\infty. $$
\end{lemma}
\begin{proof}
Fix a large integer $j>0$, and take $\vp_j: = \max\{\vp, -j \}$. 
The sequence of functions $\vp_j\in PSH(X,\omega)$ is decreasing to $\vp$,
and then it is enough to prove the integral 
$$\int_X \vp_j \omega_D^n $$
has a uniform lower bound. 

First we compute by Stoke's theorem
\begin{equation}
\label{app-0010}
\begin{split}
\int_X \vp_j (\omega+ dd^c\psi_D)^n &= \int_X \vp_j \omega\wedge\omega_D^{n-1} + \int_X \psi_D dd^c\vp_j \wedge \omega_D^{n-1}
\\
& = \int_X \vp_j  \omega\wedge\omega_D^{n-1} + \int_X \psi_D \omega_{\vp_j}\wedge \omega_D^{n-1} -  \int_X \psi_D \omega \wedge \omega_D^{n-1} 
\end{split}
\end{equation}
The third term is uniformly controlled, and the second term can be estimated as  
$$   \int_X \psi_D \omega_{\vp_j}\wedge \omega_D^{n-1} \geq \inf_X \psi_D \int_X \omega_{\vp_j}\wedge \omega_D^{n-1} \geq \inf_X \psi_D. $$
Moreover, the first term can be written as
\begin{equation}
\label{app-0011}
\int_X \vp_j \omega\wedge\omega_D^{n-1} = \int_X \vp_j \omega^2\wedge\omega_D^{n-2} + \int_X \psi_D dd^c\phi_j \wedge\omega\wedge\omega_D^{n-2}.
\end{equation}
Repeating this trick, we are able to prove 
$$ \int_X \vp_j \omega_D^n \geq \int_X \vp_j \omega^n - C, $$
for some uniform constant $C$, and our result follows. 

\end{proof}





Therefore, we proved that $PSH(X,\omega)\subset L^1(\mu)$.
Thanks to Guedj-Zeriahi's work ( Proposition (2.7), \cite{GZ}),
there exists a uniform constant $C_{\mu}>0$ such that $\forall \vp\in PSH(X,\omega)$,
\begin{equation}
\label{app-002}
-C_{\mu} + \sup_X\vp \leq \int_X \vp d\mu \leq \sup_X \vp.
\end{equation}

Before proceeding to the $\a$-invariant, we need to improve a theorem by H\"ormander (Theorem 4.45, \cite{Ho}) to the conic case.

\begin{lemma}
\label{app-lem-002}
Let $\cF$ be the family of all plurisubharmonic function $\phi$ in the unit ball $B\subset \bC^n$,
such that $\phi(0) = 0$ and $\phi(z) \leq 1$ for all $z$ close to the boundary $\d B$.
For any $\b: = \{ \b_k \}_{k=1}^d, 0<\b_k <1$, there exists two constants $ 0< r< 1$ and $C$ satisfying 
$$ \int_{B_r}  \frac{ e^{-\phi} d\lambda_{2n}}{\prod_{k=1}^d |z_k|^{2-2\b_k}} \leq C, \ \ \ \forall \phi\in \cF,$$
where $d\lambda_{2n}$ is the Lebesgue measure on $\bC^n$,
and the constants $r$ and $C$ only depend on $\b$ and $n$.
\end{lemma}
\begin{proof}
Assume $n=1$ first, and use the Green kernel to have 
\begin{equation}
\label{app-0020}
2\pi\phi = \int_{|\z|<1} \log\left|\frac{ z-\z}{1-z\bar\z} \right| d\nu(\z) + \int_{|\z|=1} \frac{1-|z|^2}{|z-\z|^2} d\sigma(\z),
\end{equation}
where the measures $d\nu = \Delta\phi \geq 0$ and $ |d\z| - d\sigma \geq 0$ on $S^1$.
Since $\phi(0)=0$, we obtain 
\begin{equation}
\label{app-0021}
\int_{|\z|<1} \log\frac{1}{|\z|} d\nu(\z) + \int_{|\z|=1} (|d\z| - d\sigma(\z)) = 2\pi.
\end{equation}
Therefore, it follows that 
\begin{equation}
\label{app-0022}
\int_{|\z|<1} \log\frac{1}{|\z|} d\nu(\z) \leq 2\pi;\ \ \ \int_{|\z|=1} |d\sigma(\z)| \leq 4\pi.
\end{equation}
Hence we have for all $|z| < e^{- \frac{1}{\b}}$,
$$ \left|   (2\pi)^{-1} \int_{|\z|=1} \frac{1-|z|^2}{|z-\z|^2} d\sigma(\z) \right| \leq 6 $$
and for any $ e^{-\frac{1}{\b}} < R < e^{- \frac{1}{2\b}}$,
$$  a: = \frac{1}{2\pi} \int_{|\z|<R} d\nu(\z) \leq \frac{1}{-\log R} < 2\b.  $$
Moreover, for such $z$ and $ |\z| > R$, we see $|z\bar\z| < |\z|^2$,
and then it is easy to see the following inequality:
$$ \left|   \frac{z-\z}{1- z\bar\z} \right| \leq \frac{1}{|\z|}. $$
This implies that 
$$ \left|   \int_{|\z|>R} \log \frac{|z-\z|}{|1-z\bar\z|} d\nu(\z)   \right| <  2\pi, \ \ \ |z|< e^{-\frac{1}{\b}}.$$
Then Jensen's inequality shows the following: 
\begin{equation}
\label{app-0023}
\begin{split}
& \exp \left(-\frac{1}{2\pi} \int_{|\z|<R} \log \frac{|z-\z|}{|1-z\bar\z|} d\nu(\z)    \right) 
\\
& = \exp \left( -\frac{1}{2\pi} \int_{|\z|<R} -a \log \frac{|z-\z|}{|1-z\bar\z|} \frac{d\nu(\z)}{2\pi a}  \right)
\\
& \leq  \frac{1}{2\pi a}  \int_{|\z|<R} \left(  \frac{|z-\z|}{|1-z\bar\z|} \right)^{-a} d\nu(\z)
\\
& < C \int_{|\z|<R} |z-\z|^{-a} d\nu(\z),
\end{split}
\end{equation}
for all $|z| < e^{-\frac{1}{\b}}$. 
Since $a<2\b$, the last term is integrable with respect to the measure $|z|^{2\b-2} d\lambda_{2n}$. 
Summing up we proved our estimate for $r:= e^{-\frac{1}{\b}}$ when $n=1$. 
For the general case, we use the polar coordinate to compute the integral 
\begin{equation}
\label{app-0024}
\int_{|z| < r } e^{-\phi (z)}d\mu(z) = \frac{1}{2\pi}\int_{|\z| =1} \frac{dS(\z)}{\prod_{k=1}^d |\z_k|^{2-2\b_k}} \int_{|w| < r } |w|^{2b-2} e^{- \phi(w\z)}d\lambda(w),
\end{equation}
where $b = n-d+ \sum_{k=1}^d \b_k > 0$. 
Taking $r(\b): = e^{-\frac{1}{b}}$, the integral is uniformly bounded.

\end{proof}

Thanks to Lemma (\ref{app-lem-002}) and equation (\ref{app-002}), 
we can follow the same argument as in Tian \cite{Tian},
and prove the existence of a conic version of the $\a$-invariants. 
\begin{prop}
\label{app-prop-001}
For all $\vp \in PSH(X,\omega)$ with $\sup_X \vp =0$, there exists a constant $\a>0$ such that 
$$\int_X e^{-\a \vp} d\mu < C, $$
for some uniform constant C only depending on $X$, $\omega$, $\mu$.
\end{prop}

Finally, for any $\cC^{2,\a,\b}$-conic K\"ahler metric $\Omega_{\b}$, the estimate in Proposition (\ref{app-prop-001}) still works 
for all $\vp\in PSH(X, \Omega_\b)$, since the conic potential is always uniformly bounded on $X$.

\end{document}